\documentclass[a4paper,leqno,10pt]{amsart}

\usepackage{epsf,graphicx,epsfig}
\usepackage{amscd}
\usepackage{amsmath,latexsym,amssymb,amsthm}
\usepackage[nospace,noadjust]{cite}
\usepackage{textcomp}
\usepackage{setspace,cite}
\usepackage{lscape,fancyhdr,fancybox}
\usepackage[all,cmtip]{xy}
\usepackage{tikz}
\usetikzlibrary{shapes,arrows,decorations.markings}
\usepackage{graphicx}

\usepackage{color}
\usepackage{url}
\usepackage{enumerate}
\usepackage[mathscr]{euscript}

  \theoremstyle{plain}

\swapnumbers
    \newtheorem{thm}{Theorem}[section]
    \newtheorem{prop}[thm]{Proposition}
   \newtheorem{lemma}[thm]{Lemma}
    \newtheorem{corollary}[thm]{Corollary}
    
    \newtheorem{subsec}[thm]{}
\theoremstyle{definition}
    \newtheorem{defn}[thm]{Definition}
    \newtheorem{exam}[thm]{Example}

\theoremstyle{remark}
     \newtheorem{remark}[thm]{Remark}

\title{}
\author{}
\date{}
\usepackage{amssymb}

\usepackage{hyperref}
\hypersetup{
	colorlinks,
	citecolor=blue,
	filecolor=black,
	linkcolor=blue,
	urlcolor=black
}

\begin{document}
\title{Nambu structures on Lie algebroids and their modular classes}

\author{Apurba Das}
\email{apurbadas348@gmail.com}
\address{Stat-Math Unit,
Indian Statistical Institute, Kolkata 700108,
West Bengal, India.}

\author{Shilpa Gondhali}
\email{shilpag@goa.bits-pilani.ac.in}
\address{Department of Mathematics,
 Birla Institute of Technology and Sciences, Pilani,
 K. K. Birla Goa Campus,
 Goa 403726, India.}

\author{Goutam Mukherjee}
\email{gmukherjee.isi@gmail.com}
\address{Stat-Math Unit,
Indian Statistical Institute, Kolkata 700108,
West Bengal, India.}

\subjclass[2010]{Primary: 17A32, 17A42.; Secondary: 53C15, 53D17.}
\keywords{Lie algebroid, Nambu-Poisson manifold, Leibniz algebra cohomology, Leibniz algebroid,\\ modular class.}

\thispagestyle{empty}

\begin{abstract}
We introduce the notion of the modular class of a Lie algebroid equipped with a Nambu structure. In particular, we recover the modular class of a Nambu-Poisson manifold $M$ with its Nambu tensor $\Lambda$ as the modular class of the tangent Lie algebroid $TM$ with Nambu structure $\Lambda.$ We show that many known properties of the modular class of a Nambu-Poisson manifold that of a Lie algebropid extend to the setting of a Lie algebroid with Nambu structure. Finally, we prove that for a large class of Nambu-Poisson manifolds considered as tangent Lie algebroids with Nambu structures, the associated modular classes are closely related to Evens-Lu-Weinstein modular classes of Lie algebroids.  
\end{abstract}
\maketitle



\vspace{0.4cm}

\section{Introduction}\label{1}
In 1973, Y. Nambu \cite{nambu} introduced the notion of Nambu structure, studied the main features of this new mechanics at the classical level and investigated the problem of quantization. In 1975, F. Bayen and M. Flato \cite{bayen-flato} observed that in the classical case this new mechanics in a three-dimensional phase space, as proposed by Nambu, is equivalent to a singular Hamiltonian mechanics. Later, in 1994, L. Takhtajan \cite{takhtajan} outlined the basic principles of a canonical formalism for the Nambu mechanics showing that it is based on the notion of a Nambu bracket, which generalizes the Poisson bracket to the multiple operation of higher order $n \geq 3.$ The author also introduced the fundamental identity (a generalization of the Jacobi identity),  as a consistency condition for the Nambu dynamics and introduced the notion of  Nambu-Poisson manifolds as phase spaces for Nambu mechanics which turned out to be more rigid than Poisson manifolds$-$phase spaces for the Hamiltonian mechanics. A smooth manifold equipped with a skew-symmetric $n$-bracket on the function algebra satisfying derivation property and the fundamental identity  is called a Nambu-Poisson manifold of order $n$. 

In \cite{wein3}, A. Weinstein introduced the notion of modular class of a Poisson manifold. Given a volume form of an oriented Poisson manifold $M$, the map which associates to a function the divergence of the corresponding Hamiltonian vector field is called the modular vector field. It turns out that it is a $1$-cocycle in the Poisson cohomology of $M$ and its class, called the modular class, is independent of the chosen volume form. The notion of modular class of Poisson manifolds  was extended to a more general context, the so called Lie algebroid with Poisson structure (\cite{yks1},\cite{yks2}). Let $A$ be a Lie algebroid. A Poisson structure on $A$ is a $2$-multisection of $A$ such that its Gerstenhaber bracket with itself vanishes. Kosmann-Schwarzbach \cite{yks1} showed that the modular class arises due to the existence of two generating operators for a BV algebra.

It is well-known that for a Nambu-Poisson manifold $M$ of order $n$, the bundle $\wedge^{n-1}T^*M$ is a Leibniz algebroid \cite{ibanez}, a non skew-symmetric analogue of a Lie algebroid. Thus the space $\Omega^{n-1}(M)$ carries a Leibniz algebra structure \cite{loday93}. In \cite{ibanez}, the authors introduced the notion of the modular class of a Nambu-Poisson manifold which generalizes the work of Weinstein \cite{wein3} for the Poisson case. This class appears as an element of the first cohomology group of the Leibniz algebra cohomology of $\Omega^{n-1}(M)$ with trivial representation in $C^\infty(M)$.

A. Wade \cite{wade} introduced the notion of a Nambu structure of order $n\geq 2$ on a Lie algebroid $A$ generalizing the notion of a Lie algebroid with Poisson structure as well as that of a Nambu-Poisson manifold and proved that like the classical case, the bundle $\wedge^{n-1}A^\ast$ is a Leibniz algebroid. Around the same time, Y. Hagiwara \cite{hagi} proposed a different definition of Nambu structure of order $n \geq 3$ on a Lie algebroid, generalizing the notion of a Nambu-Poisson manifold. It may be mentioned here that this definition makes sense even for $n= 2.$ For instance, the tangent bundle $TM$ of a Poisson manifold $(M, \pi)$ is an example according to this definition provided the Poisson tensor $\pi \in \Gamma(\wedge^2TM)$ is locally decomposable (cf. Remark \ref{equivalence-of-defn} and Proposition \ref{proof-equivalence-of-defn}).  

The aim of the present paper is to study Lie algebroid with Nambu structure and to introduce its modular class. For this we consider the definition of a Nambu structure of order $n > 2$ on a Lie algebroid originally proposed by Y. Hagiwara \cite{hagi}.  Roughly speaking, given a Lie algebroid $A$, an $n$-multisection $\Pi \in \Gamma(\wedge^{n}A)$ $(\mbox{rank}~A \geq  n > 2)$ is called a Nambu structure on $A$ of order $n$ if it satisfies certain condition (cf. Definition \ref{Def}). A Nambu structure on a Lie algebroid $A$ is called a  maximal Nambu structure if the order of the Nambu structure is equal to the rank of $A$. It is known that Nambu tensor of a Nambu-Poisson manifold of order $> 2 $ is locally decomposable (\cite{duf-zung}, \cite{ibanez}). For Lie algebroids with Nambu structures, the same result holds true provided the space of sections $\Gamma{A^*}$ is locally generated by elements of the form $d_Af$, $f \in C^\infty(M)$, where $d_A$ is the differential of the cochain complex with trivial coefficients associated to the Lie algebroid $A$ (cf. Lemma 3.9, \cite{hagi}). We prove that Definition \ref{Def} considered in the present paper coincides with the one defined in \cite{wade} when the Nambu structure is locally decomposable (Proposition \ref{proof-equivalence-of-defn}). 

Next, we show that given a Lie algebroid $A$ with a Nambu structure of order $n$, the bundle $\wedge^{n-1}A^*$ has a Leibniz algebroid structure. This Leibniz algebroid structure is distinct from the one obtained in \cite{wade} and it reduces to the Leibniz algebroid obtained in \cite{ibanez}, in the case of the tangent Lie algebroid $A = TM$ of a Nambu-Poisson manifold $M$ (cf. Remark \ref{leibniz algebra-comparison}). It turns out that when the Nambu structure is maximal, the associated Leibniz algebroid is a Lie algebroid (cf. Proposition \ref{leib-lie}). It may be remarked that in \cite{grabowski-poncin} (see also \cite{jubin-poncin-uchino}), the authors introduced a notion of Loday algebroid which is more geometric than Leibniz algebroid and appears in many places. Lie algebroids, Courant algebroids, Leibniz algebroids associated to  Nambu-Poisson manifolds are examples of Loday algebroids. We observe in this paper that the above mentioned Leibniz algebroid $\wedge^{n-1}A^*$ associated to a Lie algebroid $A$ with a Nambu structure turns out to be a Loday algebroid (see Remark \ref{loday-algebroid}).

For an oriented Lie algebroid $A$ with a Nambu structure of order $n$, we define the notion of a modular tensor field $M^\mu \in \Gamma(\wedge^{n-1}A)$ associated to a chosen non-vanishing section $\mu \in \Gamma(\wedge^\text{top} A^*)$ giving the orientation (cf. Definition \ref{mod-Def}). We show that the modular tensor field defines a $1$-cocycle in the Leibniz algebra cohomology of $\Gamma(\wedge^{n-1}A^*)$ with coefficients in $C^\infty(M)$ and the corresponding cohomology class does not depend on the chosen section. We call this cohomology class the {\it modular class} of the Lie algebroid with the given Nambu structure (Definition \ref{def-modular-class}). We show that as in the classical case, this modular class can be viewed as an obstruction to the existence of a density invariant under all Hamiltonian $A$-sections (cf. Proposition \ref{modular-class-obstruction}).

It may be remarked that if we start with the definition of a Lie algebroid with a Nambu structure as introduced in \cite{wade} and assume that $A$ is oriented with $\mu \in \Gamma (\wedge^{top}A^\ast),$ a nowhere vanishing element representing the orientation, then, as before we may define the notion of a modular tensor field $M^{\mu} \in \Gamma (\wedge^{n-1}A)$ associated to $\mu.$ However, it is not clear why $M^\mu$ should be a $1$-cocycle in the Leibniz algebra cohomology of the Leibniz algebra $\Gamma(\wedge^{n-1}A^\ast)$ considered by Wade.  

In \cite{elw}, the authors introduced the notion of characteristic class of a Lie algebroid $A$ with a representation on a line bundle $L$ and used it to define the notion of modular class of a Lie algebroid. We show that for a large class of Nambu-Poisson manifolds, the notion of modular class is closely related to the notion of modular class introduced in \cite{elw}.

The paper is organized as follows. In Section \ref{2}, we recall some basic definitions and results that will be used in the paper. In Section \ref{3}, we recall the definition of Nambu structure of order $n >2$ on a Lie algebroid proposed by Y. Hagiwara and provide a collection of examples. Then, we discuss decomposability of Nambu structure and compare our definition with the one introduced in \cite{wade}. We prove that there is a Leibniz algebroid, which turns out to be a Loday algebroid also, associated to a given Lie algebroid with Nambu structure. In Section \ref{4}, we introduce the notion of a modular tensor field of a Lie algebroid with a Nambu structure, study its properties and define the modular class of a Lie algebroid with Nambu structure. We also compute modular tensor field and modular class for a class of examples. Finally, in Section \ref{5}, we prove that for a large class of Nambu-Poisson manifolds considered as tangent Lie algebroids with Nambu structures, the associated modular classes are closely related to Evens-Lu-Weinstein modular classes of Lie algebroids \cite{elw} (cf. Theorem \ref{modular-compare}). 

\section{Preliminaries}\label{2}
In this section, we recall some definitions, fix notations which we use throughout the paper and adapt a result known for manifolds \cite{illmp}, in the context of oriented Lie algebroids.

The notion of a Nambu-Poisson manifold is a  higher order generalization of a Poisson manifold and is defined as follows (\cite{duf-zung}, \cite{ibanez}, \cite{takhtajan}).
\begin{defn}
Let $M$ be a smooth manifold. A {\it Nambu-Poisson structure of order $n$} on $M$ is a skew-symmetric $n$-multilinear mapping
\begin{align*}
 \{~, \ldots, ~\} : C^\infty (M) \times \cdots \times C^\infty (M) \rightarrow C^\infty (M)
\end{align*}
satisfying the following conditions:
\begin{enumerate}
 \item { Leibniz rule}:\\
  $\{f_1, \ldots, f_{n-1}, fg\} = f \{f_1, \ldots, f_{n-1}, g \} + \{f_1, \ldots, f_{n-1}, f \} g$;
 \item { Fundamental identity}:\\
 $\{f_1, \ldots ,f_{n-1}, \{g_1,\ldots, g_n\}\} = \sum_{i=1}^n  \{g_1,\ldots,g_{i-1},\{f_1,\ldots,f_{n-1},g_i\},\ldots, g_n\}$,
\end{enumerate}
for all $f_i, g_j, f, g \in C^\infty (M).$ A manifold together with a Nambu-Poisson structure as above is called a {\it Nambu-Poisson manifold of order $n$}.
\end{defn}

A Nambu-Poisson manifold of order $2$ is nothing but a Poisson manifold \cite{vais-book}. Since the bracket above is skew-symmetric and satisfies Leibniz rule, there exists an $n$-vector field $\Lambda \in \Gamma({\wedge^n TM})$ such that $\{f_1, \ldots, f_n\} = \Lambda (df_1, \ldots, df_n),$ 
for all $ f_1, \ldots, f_n \in C^\infty (M)$. Given any $(n-1)$ functions $ f_1, \ldots, f_{n-1} \in C^\infty (M)$, the {\it Hamiltonian vector field} $X_{f_1...f_{n-1}}$
associated to these functions is defined by $X_{f_1...f_{n-1}} = \Lambda^{\sharp}(df_1 \wedge \ldots \wedge df_{n-1}),$ where $\Lambda^{\sharp} : \wedge^{n-1}T^*M \rightarrow TM,~~ \omega \mapsto \iota_{\omega}\Lambda$ is the bundle map induced by $\Lambda$, $\iota$ being the contraction operator. In other words, the Hamiltonian vector field is defined by $ X_{f_1...f_{n-1}} (g) = \{f_1, \ldots ,f_{n-1}, g\},~~\mbox{for}~~ g \in C^\infty(M).$

In terms of the Lie derivative operator $\mathcal L$, the fundamental identity can be rephrased as
\begin{align}\label{lie-der-nambu}
 \mathcal{L}_{X_{f_1...f_{n-1}}} \Lambda = 0,
\end{align}
for all $f_1, \ldots, f_{n-1} \in C^\infty(M),$  which shows that every Hamiltonian vector field preserves the Nambu tensor. A Nambu Poisson manifold is often denoted by $(M, \{~, \ldots, ~\})$ or simply by $(M, \Lambda)$ . Moreover, note that the fundamental identity is equivalent to
$$[X_{f_1...f_{n-1}}, X_{g_1...g_{n-1}}] = \sum_{i=1}^{n-1} X_{g_1... g_{i-1}\{f_1, \ldots , f_{n-1}, g_i\}...g_{n-1}},$$ where the bracket on the left is the Lie algebra bracket of vector fields.

The following result describes the local structure of a Nambu-Poisson manifold (\cite{duf-zung}, \cite{ibanez}).
\begin{thm}\label{local-structure-NP-manifold}
Let $M$ be a smooth manifold of dimension $m$. An $n$-vector field $\Lambda$ ($3 \leq n \leq m$) defines a Nambu-Poisson structure on $M$ if and only if for all $x \in M$ with $\Lambda (x) \neq 0,$ there exists a local chart $(U; x_1, \ldots ,x_n, x_{n+1}, \ldots , x_m)$ around $x$ such that $\Lambda|_U = \frac{\partial}{\partial x_1} \wedge \cdots \wedge \frac{\partial}{\partial x_n}.$  
\end{thm}

\begin{remark}\label{singular-foliation-NP}
The Nambu-Poisson structure is closely related to singular foliations. Let $(M, \Lambda)$ be a Nambu-Poisson manifold of order $n$. Recall that (Chapter 6, \cite{duf-zung}) there is a (singular) distribution $\mathcal D (M)= \cup_{m\in M}\mathcal D_m(M)$ on $M,$ where $\mathcal D_m (M)\subset T_m M$ is the subspace of the tangent space at $m$ generated by all Hamiltonian vector fields on $M.$ It is clear from the above expression of the fundamental identity that this distribution is integrable and its leaves are either $n$-dimensional submanifolds endowed with a volume form or just singletons. 
\end{remark}  

Recall that the notion of Lie algebroid is a generalization of tangent bundles of  manifolds as well as Lie algebras \cite{mac-book}.  

\begin{defn}\label{lie-algbd}
A {\it Lie algebroid} $(A, [~, ~], \rho)$ over a smooth manifold $M$ is a smooth vector bundle $A$ over
$M$ together with a Lie
algebra structure $[~, ~]$ on the space $\Gamma{A}$ of the smooth sections of $A$ and a bundle map $\rho : A \rightarrow T M $ , called the {\it anchor}, such that
\begin{enumerate}
\item $[X, f Y ] = f [X, Y ] + (\rho(X)f)Y$,
\item $\rho([X, Y]) = [\rho (X), \rho (Y)],$
\end{enumerate}
for all $X,~ Y \in \Gamma{A} $ and $f \in C^\infty (M)$.
\end{defn}

Any Lie algebra can be considered as a Lie algebroid over a point. The tangent bundle of a smooth manifold is a Lie algebroid with the usual Lie bracket of vector fields. For any Poisson manifold, its cotangent bundle carries a natural Lie algebroid structure (\cite{mac-book}, \cite{vais-book}).

Recall that given a Lie algebroid $(A, [~, ~], \rho)$, the Lie bracket on $\Gamma A$ can be extended to the so called Schouten bracket $[~,~]$ on the exterior algebra $\Gamma (\wedge^ {\bullet} A)$ of multisections of $A.$ The space $\Gamma (\wedge^ {\bullet} A)$ together with this extended bracket form a Schouten algebra (also called Gerstenhaber algebra). It is characterized by the following properties: for $f \in C^\infty(M)$, $X, Y \in \Gamma A$ and $P, Q, R  \in \Gamma (\wedge^ {\bullet} A),$
\begin{enumerate}\label{schouten-bracket}
\item $[P, Q] \in \Gamma (\wedge^{|P| +|Q|-1}A);$
\item $[X, f]= \rho (X)(f);$
\item $[X, Y]$ is the Lie bracket on $\Gamma A;$
\item $[P, Q] = -(-1)^{(|P|-1)(|Q|-1)}[Q, P];$
\item $[P, Q\wedge R]= [P,Q] \wedge R +(-1)^{(|P|-1)|Q|}Q\wedge [P,R],$\\ where $|P|$ denotes the degree of $P$.
\end{enumerate}  

Moreover, $\Gamma(\wedge^{\bullet} A^*)$ together with the Lie algebroid differential $d_A$ forms a differential graded algebra, where the differential $d_A$ has the explicit formula similar to the de Rham differential formula
\begin{align*}
 (d_A \alpha)(X_0, X_1, \ldots, X_n) =& \sum_{i=0}^n (-1)^{i} \rho (X_i) \alpha(X_0, \ldots , \hat {X_i}, \ldots , X_n) \\
& + \sum_{i < j} (-1)^{i+j} \alpha ([X_i,X_j], X_0, \ldots , \hat {X_i},\ldots , \hat {X_j}, \ldots , X_n),
\end{align*}
where $\alpha \in \Gamma(\wedge ^n A^*)$ and $X_0, \ldots, X_n \in \Gamma A$ and $\hat{X_i}$ in any term above means that $X_i$ is missing from the expression. For $A= TM$, the usual tangent bundle Lie algebroid, the differential $d_A$ turns out to be the exterior differential $d$ of the de Rham complex of $M$. 

For $\alpha, ~\beta \in \Gamma (\wedge^\bullet A^\ast),$ the Lie algebroid differential $d_A$ satisfies
$$d_A(\alpha\wedge \beta) = d_A(\alpha) \wedge \beta + (-1)^{|\alpha|}\alpha \wedge d_A (\beta).$$ 

Next, we briefly recall the definitions and properties of the contraction operators and the Lie derivative operators, details may be found in \cite{ elw}, \cite{mac-book}. Let $A$ be a Lie algebroid. Let $X \in \Gamma A.$
\begin{defn}\label{Lie derivative-contraction}
\begin{enumerate}
\item The contraction operator $\iota_X : \Gamma (\wedge^\bullet A^\ast) \rightarrow \Gamma (\wedge^{\bullet-1} A^\ast)$ is defined as follows:
for $\alpha \in \Gamma (\wedge^n A^\ast)$ and $ X_1, \ldots, X_{n-1} \in \Gamma A,$
$$\iota_X\alpha(X_1, \ldots, X_{n-1}) = \alpha (X, X_1, \ldots, X_{n-1}).$$  For $X_1, \ldots, X_k \in \Gamma A,$ the contraction with respect to the monomial $X_1\wedge \cdots \wedge X_k$ is then defined  by $ \iota _{X_1\wedge \cdots \wedge X_k} = \iota_{X_k} \circ \cdots \circ \iota_{X_1}.$ This extends to give a well-defined map $\iota_P: \Gamma (\wedge^\bullet A^\ast) \rightarrow \Gamma(\wedge^{\bullet-k}A^\ast),$ for $P \in \Gamma(\wedge^kA).$

\item For any $\alpha \in \Gamma A^\ast,$ the contraction $\iota_{\alpha} : \Gamma(\wedge^\bullet A) \rightarrow \Gamma(\wedge^{\bullet-1}A)$ is defined by 
$$\iota_{\alpha}P (\beta) = P(\alpha\wedge \beta),~~P \in \Gamma(\wedge^\bullet A),~~\beta \in \Gamma (\wedge^{\bullet-1}A^\ast).$$
For $\alpha_r \in \Gamma A^\ast,~~r = 1, \ldots, k,$ the contraction with respect to $\alpha_1 \wedge \cdots \wedge \alpha_k$ is then defined by $\iota_{\alpha_1 \wedge \cdots \wedge\alpha_k}= \iota_{\alpha_k}\circ \cdots \circ \iota_{\alpha_1}.$  This extends to a well-defined map $\iota_{\eta}: \Gamma(\wedge^\bullet A) \rightarrow \Gamma(\wedge^{\bullet-k}A),$ for any $\eta \in \Gamma(\wedge^k A^\ast).$ 

\item The Lie derivative $\mathcal L_X : \Gamma (\wedge^\bullet A^\ast) \rightarrow \Gamma (\wedge^\bullet A^\ast)$ is a degree zero map and is defined by the following formula
$$ \mathcal{L}_X(\alpha)(X_1,\ldots, X_n)= \mathcal L_X(\alpha(X_1,\ldots, X_n))
-\displaystyle{\sum_{i=1}^n}\alpha(X_1,\ldots,[X,X_i],\ldots,X_n)$$
where $\alpha \in\Gamma (\wedge^n A^\ast)$ and $X_1, \ldots X_n \in \Gamma A.$ 
More generally, for $P \in \Gamma(\wedge^iA)$ the generalized Lie derivative $\mathcal L_P : \Gamma(\wedge^\bullet A^\ast) \rightarrow \Gamma(\wedge^{\bullet-i+1}A^\ast)$ with respect to the multisection $P$ is defined by the bracket for graded endomorphisms as follows:
$$\mathcal L_P = [\iota_P, d_A] = \iota_P\circ d_A -(-1)^id_A \circ \iota_P.$$  

\item The Lie derivative 
$$\mathcal L_X : \Gamma (\wedge^\bullet A) \rightarrow \Gamma (\wedge^\bullet A)$$ is defined by the Schouten bracket on $\Gamma (\wedge^\bullet A),$ namely, $\mathcal L_X (P) = [X, P],$ for $P \in \Gamma (\wedge^\bullet A).$ Explicitly, for $P \in \Gamma (\wedge^nA)$ and $\alpha_1, \ldots, \alpha_n \in \Gamma A^\ast,$
$$(\mathcal L_XP)(\alpha_1, \ldots , \alpha_n) = \mathcal L_X(P(\alpha_1, \ldots , \alpha_n)) - \sum_{i=1}^nP(\alpha_1, \ldots , \mathcal L_X(\alpha_i), \ldots, (\alpha_n)).$$
\end{enumerate}
\end{defn}

The contraction operators as defined above satisfy the following properties:
$$\iota_X(f\alpha) = f\iota_X(\alpha),~~\iota_{fX}(\alpha) = f\iota_X(\alpha),~~ \iota_X\circ \iota_Y = -\iota_Y\circ\iota_X,~~\alpha \in \Gamma(\wedge^\bullet A^\ast).$$

\begin{prop}\label{properties-Lie derivative-contraction}
The Lie derivative operators satisfy the following properties. Let  $X, Y \in \Gamma A,$ $P, Q \in \Gamma(\wedge^\bullet A),$ $\alpha, \beta \in \Gamma(\wedge^\bullet A^\ast)$ and $f \in C^\infty(M).$
\begin{enumerate}
\item $\mathcal L _X(P \wedge Q) = \mathcal L_X( P) \wedge Q + P \wedge \mathcal L _X(Q).$
\item $[\mathcal L _X, \mathcal L _Y](P) := (\mathcal L _X\circ\mathcal L _Y-\mathcal L _Y \circ \mathcal L _X)(P) = \mathcal L _{[X, Y]}(P).$
\item $\mathcal L _X (fP) = f\mathcal L _X(P) +\rho(X)(f)P.$
\item $\mathcal L _{fX} (P) = f\mathcal L _X(P)- X \wedge \iota_{d_Af}(P).$
\item $\mathcal L _X (f\alpha) = f\mathcal L _X(\alpha) +\rho(X)(f)\alpha.$
\item $\mathcal L_{fX}\alpha = f \mathcal L_X\alpha + d_Af \wedge \iota_X\alpha.$
\item $\mathcal L _{[X, Y]}(\alpha) = (\mathcal L _X\circ\mathcal L _Y-\mathcal L _Y \circ \mathcal L _X)(\alpha).$
\item $[\mathcal L _X, \iota_Y] :=  \mathcal L _X\circ \iota_Y-\iota_Y\circ \mathcal L _X = \iota_{[X, Y]}.$
\item $\mathcal L_X = \iota_X d_A + d_A\iota_X.$
\item $\mathcal L_X \circ d_A = d_A \circ \mathcal L_X.$
\end{enumerate}
\end{prop}

\begin{remark}\label{properties-top-degree}
In particular, for $f \in C^\infty(M)$, $X\in \Gamma A$, $P \in \Gamma(\wedge^{\text{top}}A)$ and $\mu \in \Gamma(\wedge^{\text{top}}A^\ast)$
\begin{enumerate}
\item $\mathcal L_{fX}P = f\mathcal L_XP - \rho(X)(f)P.$
\item $\mathcal L_X(fP) = f\mathcal L_XP + \rho(X)(f)P.$
\item $\mathcal L_{fX}(\mu) = f\mathcal L_X\mu + \rho(X)(f)\mu.$
\item $\mathcal L_X(f\mu) = f\mathcal L_X\mu + \rho(X)(f)\mu.$
\end{enumerate}
\end{remark}  

Next, we adapt a result of \cite{illmp} in the context of oriented Lie algebroid which will  be used in the subsequent sections.
 
Let $A$ be a Lie algebroid of rank $m$. Assume that $A$ is oriented as a vector bundle. Let $\mu \in \Gamma (\wedge^mA^*)$ be a nowhere vanishing element giving the orientation.
Then for each $0 \leq k \leq m,$ $\mu$ defines a vector bundle isomorphism $*_\mu : \wedge^kA \rightarrow \wedge^{m-k}A^*$ given by   
$$*_\mu (P) = \iota_{P}\mu,~~P \in \Gamma (\wedge^kA)~~\text{and}~~*_\mu (f) = f \mu,~~f \in C^\infty (M).$$

Define an operator $\partial_{\mu}$ by
\begin{align}\label{partial}
\partial_\mu := {*_\mu}^{-1} \circ d_A \circ {*_\mu} :  \Gamma(\wedge^kA) \rightarrow  \Gamma(\wedge^{k-1}A).
\end{align}
The operator $\partial_\mu$  satisfies the condition $\mathcal{L}_X \mu = \partial_\mu (X) \mu, ~~X \in \Gamma A,$ and hence, may be  called the divergence with respect to $\mu$. Therefore,  for $X \in \Gamma{A},$ we may write $\partial_\mu (X) =~~\mbox{ div}_\mu X$. Clearly ${\partial_\mu}^2 = 0$. The homology of the complex $(\Gamma(\wedge^{\bullet}A), \partial_\mu )$ is denoted by $\mathcal{H}^{\mu}_{\bullet}(A).$
\begin{lemma}
 Let $A$ be an oriented Lie algebroid and $\mu \in \Gamma(\wedge^\text{top}A^*)$ be a nowhere vanishing element. Then $\mathcal{L}_X (*_\mu P )= *_\mu (\mathcal{L}_XP) + (\text{div}_\mu X) *_\mu{P},$ where $P \in \Gamma(\wedge^{\bullet}A)$ and $X \in \Gamma{A}.$
\end{lemma}

\begin{proof}
 Let $f \in C^\infty(M).$  Then, by Remark \ref{properties-top-degree}(4) we have
$$\mathcal{L}_X (*_\mu f ) = \mathcal{L}_X (f \mu) = f \mathcal{L}_X ( \mu) + (\rho (X)f) \mu = f (\text{div}_\mu X) \mu + *_\mu (\rho (X)f) = (\text{div}_\mu X) *_\mu (f) + *_\mu (\mathcal{L}_X f).$$
Let $P \in \Gamma{A}.$ Then,
\begin{align*}
 \mathcal{L}_X (*_\mu P ) =& \mathcal{L}_X \iota_P \mu = \mathcal{L}_X \iota_P \mu - \iota_P \mathcal{L}_X \mu + \iota_P \mathcal{L}_X \mu = \iota_{[X, P]} \mu + \iota_P (\text{div}_\mu X) \mu \\
=& \iota_{\mathcal{L}_X P} \mu + (\text{div}_\mu X) \iota_P \mu = *_\mu (\mathcal{L}_X P) + (\text{div}_\mu X) *_\mu (P).
\end{align*}
Thus the result is true for smooth functions on $M$ and sections of $A$. Since both sides satisfy derivation property with respect to $P$, therefore, they depend only
on the local behavior of $P$. By induction, it follows that the result is true for decomposable multisections. Since any arbitrary multisection $P \in \Gamma(\wedge^{\bullet}A)$ locally is a finite sum of decomposable multisections, the result follows.
\end{proof}

Next proposition will be useful throughout the paper.

\begin{prop}\label{formula}
 Let $A$ be an oriented Lie algebroid and $\mu \in \Gamma(\wedge^\text{top}A^*)$  a nowhere vanishing section inducing the orientation. Then for any $P \in \Gamma(\wedge^kA)$ and $\alpha \in \Gamma(\wedge^{k-1}A^*)$, $k \geq 1$, we have $\iota_{\alpha} \partial_\mu (P) = \text{div}_\mu (\iota_\alpha P) + (-1)^k \iota_{d_A \alpha} P.$
\end{prop}

\begin{proof}
We use induction to prove the result. 

Let $X \in \Gamma A$ and $f \in C^\infty(M).$ By Remark \ref{properties-top-degree}(3), we have
$$\mathcal{L}_{f X} \mu = f \mathcal{L}_{ X} \mu + \rho (X)(f) \mu.$$
In other words, $\partial_\mu (f X) \mu = f \partial_\mu ( X) \mu + \rho (X)(f) \mu.$ Therefore,
$$f \partial_\mu ( X) = \partial_\mu (f X) - \rho (X)(f),~~\text{or}, \iota_f \partial_\mu ( X) = \text{div}_\mu (f X) - \iota_{d_A f} X,$$
proving the result  for $k=1$. Assume that the result is true for $1\leq j\leq (k-1).$  We claim that the result holds for any decomposable $k$-multisection.

Let $P = X_1 \wedge \cdots \wedge X_k$ be a decomposable $k$-multisection.
Let $\beta \in \Gamma(\wedge^{k-2}A^*).$ By induction hypothesis, we have
\begin{align*}
 &  \iota_\beta \partial_\mu (X_1 \wedge \cdots \wedge X_{k-1})\\
&=  \text{div}_\mu \big(\iota_\beta (X_1 \wedge \cdots \wedge X_{k-1})\big) + (-1)^{k-1} \iota_{d_A \beta} (X_1 \wedge \cdots \wedge X_{k-1})\\
&=  \sum_{i=1}^{k-1} (-1)^{k-1 +i} \text{div}_\mu \big(  \beta(X_1 \wedge \cdots \wedge \hat{X_i} \wedge \cdots \wedge X_{k-1}) X_i  \big)\\
& \hspace*{4cm}+ (-1)^{k-1} {d_A \beta} (X_1 \wedge \cdots \wedge X_{k-1})\\
&= \sum_{i=1}^{k-1} (-1)^{k-1 +i} \text{div}_\mu \big(  \beta(X_1 \wedge \cdots \wedge \hat{X_i} \wedge \cdots \wedge X_{k-1}) X_i  \big)\\
&+  (-1)^{k-1} \sum_{i=1}^{k-1} (-1)^{i-1} \rho(X_i) \beta (X_1, \ldots, \hat{X_i}, \ldots, X_{k-1})\\
&+  (-1)^{k-1} \sum_{1 \leqslant i < j \leqslant k-1} (-1)^{i+j} \beta ([X_i, X_j], X_1, \ldots, \hat{X_i}, \ldots, \hat{X_j}, \ldots, X_{k-1})\\
&= \sum_{i=1}^{k-1} (-1)^{k-1 +i}  \beta(X_1 \wedge \cdots \wedge \hat{X_i} \wedge \cdots \wedge X_{k-1}) \text{div}_\mu (X_i) \\
&+  (-1)^{k-1} \sum_{1 \leqslant i < j \leqslant k-1} (-1)^{i+j} \beta ([X_i, X_j], X_1, \ldots, \hat{X_i}, \ldots, \hat{X_j}, \ldots, X_{k-1}),
\end{align*}
since  $\text{div}_\mu (fX) = f \text{div}_\mu (X) + \rho(X) f.$
Hence,
\begin{align*}
& \partial_\mu (X_1 \wedge \cdots \wedge X_{k-1}) \\=& (-1)^{k-1} \sum_{i=1}^{k-1} (-1)^i\big(\text{div}_\mu (X_i)\big) X_1 \wedge \cdots \wedge \hat{X_i} \wedge \cdots \wedge X_{k-1}\\
&+ (-1)^{k-1} \sum_{1 \leqslant i < j \leqslant k-1} (-1)^{i+j} [X_i, X_j]\wedge X_1 \wedge \cdots  \wedge\hat{X_i}\wedge  \cdots  \wedge\hat{X_j}\wedge  \cdots \wedge X_{k-1}.\hspace{2mm}{\bf(Eqn~ A)} 
\end{align*}
Therefore,
\begin{align*}
& d_A ( *_\mu(P)) = d_A (\iota_{X_1 \wedge \cdots \wedge X_k} \mu) = d_A (\iota_{X_k} \iota_{X_{k-1}} \ldots \iota_{X_1} \mu) = d_A (\iota_{X_k} *_\mu (X_1 \wedge \cdots \wedge X_{k-1}))\\
=& \mathcal{L}_{X_k} *_\mu (X_1 \wedge \cdots \wedge X_{k-1}) - \iota_{X_k} d_A *_\mu (X_1 \wedge \cdots \wedge X_{k-1}) \\
=& \mathcal{L}_{X_k} *_\mu (X_1 \wedge \cdots \wedge X_{k-1}) - \iota_{X_k} *_\mu (\partial_\mu (X_1 \wedge \cdots \wedge X_{k-1}))\\ 
=&  *_\mu \mathcal{L}_{X_k} (X_1 \wedge \cdots \wedge X_{k-1}) + (\text{div}_\mu X_k) *_\mu (X_1 \wedge \cdots \wedge X_{k-1})\\
& - \iota_{X_k} *_\mu (\partial_\mu (X_1 \wedge \cdots \wedge X_{k-1}))\\
=& *_\mu \bigg( \sum_{i=1}^{k-1} X_1 \wedge \cdots \wedge [X_k, X_i] \wedge \cdots \wedge X_{k-1} \bigg) + (\text{div}_\mu X_k) *_\mu (X_1 \wedge \cdots \wedge X_{k-1})\\
& - \iota_{X_k} *_\mu (\partial_\mu (X_1 \wedge \cdots \wedge X_{k-1}))\\
=& *_\mu (\sum_{i=1}^{k-1} (-1)^i [X_i, X_k] \wedge X_1 \wedge \cdots \wedge \hat{X_i} \wedge \cdots \wedge X_{k-1}) + (\text{div}_\mu X_k) *_\mu (X_1 \wedge \cdots \wedge X_{k-1})\\ 
&+ (-1)^k \iota_{X_k} *_\mu \bigg( \sum_{i=1}^{k-1}(-1)^i (\text{div}_\mu X_i) X_1 \wedge \cdots \wedge \hat{X_i} \wedge \cdots \wedge X_{k-1}\\
&+  \sum_{1 \leq i < j \leq k-1} (-1)^{i+j} [X_i, X_j]\wedge X_1 \wedge \cdots  \wedge\hat{X_i}\wedge  \cdots  \wedge\hat{X_j}\wedge  \cdots \wedge X_{k-1}\bigg)~~ (\text{by} ~~ {\bf(Eqn~ A)})\\
=& *_\mu (\sum_{i=1}^{k-1} (-1)^i [X_i, X_k] \wedge X_1 \wedge \cdots \wedge \hat{X_i} \wedge \cdots \wedge X_{k-1}) + (\text{div}_\mu X_k) *_\mu (X_1 \wedge \cdots \wedge X_{k-1}) \\
&+ (-1)^k \sum_{i=1}^{k-1} (-1)^i(\text{div}_\mu X_i) *_\mu (X_1 \wedge \cdots \wedge \hat{X_i} \wedge \cdots \wedge X_{k-1} \wedge X_k)\\
&+ (-1)^k  \sum_{1 \leq i < j \leq k-1} (-1)^{i+j} *_\mu ([X_i, X_j]\wedge X_1 \wedge \cdots  \wedge\hat{X_i}\wedge  \cdots  \wedge\hat{X_j}\wedge  \cdots \wedge X_{k-1} \wedge X_k).
\end{align*}
Thus,
\begin{align*}
 (-1)^k d_A ( *_\mu(P)) =& *_\mu \bigg( \sum_{i=1}^{k}(-1)^i (\text{div}_\mu X_i) (X_1 \wedge \cdots \wedge \hat{X_i} \wedge \cdots  \wedge X_k)\\
&+ \sum_{1 \leq i < j \leq k} (-1)^{i+j} [X_i, X_j]\wedge X_1 \wedge \cdots  \wedge\hat{X_i}\wedge  \cdots  \wedge\hat{X_j}\wedge  \cdots \wedge X_k \bigg),
\end{align*}
or,
\begin{align*}
 (-1)^k \partial_\mu (P) =& \sum_{i=1}^{k}(-1)^i (\text{div}_\mu X_i) (X_1 \wedge \cdots \wedge \hat{X_i} \wedge \cdots  \wedge X_k)\\
&+ \sum_{1 \leq i < j \leq k} (-1)^{i+j} [X_i, X_j]\wedge X_1 \wedge \cdots  \wedge\hat{X_i}\wedge  \cdots  \wedge\hat{X_j}\wedge  \cdots \wedge X_k.\hspace{2mm}{\bf(Eqn~ B)} 
\end{align*}
Now for any $\alpha \in \Gamma(\wedge^{k-1}A^*)$, we have
\begin{align*}
&  (-1)^k \text{div}_\mu (\iota_\alpha P) + \iota_{d_A \alpha} P\\
= & (-1)^k \text{div}_\mu (\iota_\alpha (X_1 \wedge \cdots \wedge X_k)) + d_A \alpha (X_1, \ldots, X_k)\\
= & (-1)^k \sum_{i=1}^k (-1)^{i+k}  \text{div}_\mu (\alpha(X_1, \ldots, \hat{X_i}, \ldots, X_k) X_i)
\end{align*}
\begin{align*}
&+  \sum_{i=1}^k (-1)^{i+1} \rho (X_i) \alpha (X_1, \ldots, \hat{X_i}, \ldots, X_k)\\
&+  \sum_{1 \leq i < j \leq k} (-1)^{i+j} \alpha ([X_i, X_j]\wedge X_1 \wedge \cdots  \wedge\hat{X_i}\wedge  \cdots  \wedge\hat{X_j}\wedge  \cdots \wedge X_k)\\
=&  \sum_{i=1}^k (-1)^i \alpha (X_1, \ldots, \hat{X_i}, \ldots, X_k) (\text{div}_\mu X_i)\\
&+  \sum_{1 \leq i < j \leq k} (-1)^{i+j} \alpha ([X_i, X_j], X_1 ,\ldots , \hat{X_i} , \ldots , \hat{X_j} , \ldots , X_k)\\
=&  (-1)^k \iota_\alpha \partial_\mu (P), ~~ (\text{by} ~~ {\bf(Eqn~ B)})
\end{align*}
as $\text{div}_\mu (f X) = f \text{div}_\mu ( X) + \rho(X) f.$
Thus, the result holds for decomposable $k$-multisections. As the expressions on both sides of the equality depend only on the local behavior of $P$, it follows that the result
holds for any multisection.
\end{proof}

Finally, in this section,  we recall the notion of a Leibniz algebroid (\cite{ibanez}), which is a non skew-symmetric analogue of the notion of a Lie algebroid, and the definition of Leibniz algebroid cohomology with coefficients in the trivial representation. To simplify notation, we shall denote a Lie algebra bracket or a Leibniz algebra bracket by the same symbol $[~, ~],$ without causing any confusion.  

\begin{defn}\label{Def-leib-algbd}
 A {\it (left) Leibniz algebroid} over $M$ is a vector bundle $A$ over $M$ together with a bracket $[~,~]$ on the space $\Gamma{A}$ of smooth sections of $A$ and a bundle map $\rho : A \rightarrow TM$, called the {\it anchor} 
such that the bracket satisfies
\begin{enumerate}
\item (left) Leibniz identity: $ [X, [Y, Z]] = [[X, Y], Z] + [Y, [X, Z]],~~X, Y, Z \in \Gamma{A};$
\item $[X,fY] = f[X, Y] + (\rho(X)f)Y$;
\item $\rho ([X, Y]) = [\rho (X), \rho (Y)], ~~X, Y \in \Gamma A,~~f\in C^\infty(M).$
\end{enumerate}
\end{defn}
Definition of  morphism between Leibniz algebroids over the same base is similar to that of Lie algebroids.

Observe that the space $\Gamma{A}$ together with the bracket $[~,~]$ forms a (left) Leibniz algebra (\cite{loday93, loday}). Moreover, from  condition (3) of the above definition, the Leibniz algebra
$(\Gamma{A}, [~,~])$ has a representation on $C^\infty(M)$ given by
$$\Gamma A \times C^\infty(M) \rightarrow C^\infty(M),~~(X, f) \mapsto \rho(X)(f), ~~X \in \Gamma A,~ f \in C^\infty(M),$$
called the {\it trivial representation}. 

The Leibniz algebroid cohomology of a Leibniz algebroid $\mathcal A = (A, [~,~], \rho)$ is the cohomology of the Leibniz algebra $(\Gamma A,[~,~])$ with coefficients in the trivial representation.
\begin{defn}\label{Def-leib-algbd-cohomology}
Let 
$$ C^k(\Gamma A; C^\infty(M)) = \{ c^k : \Gamma A \times \stackrel{k}{\cdots}\times \Gamma A \rightarrow C^\infty(M)|~~ c^k ~~\mbox{is $\mathbb R$-multilinear}\}.$$ Let 
$$d_{\mathcal A}: C^k(\Gamma A; C^\infty(M))\rightarrow  C^{k+1}(\Gamma A; C^\infty(M))$$ be the $\mathbb R$-linear map defined by
\begin{align*}
&d_{\mathcal A}c^k(X_0, X_1, \ldots , X_k) = \sum_{i= 0}^k (-1)^i\rho(X_i)(c^k(X_0, \ldots, \widehat{X}_i, \ldots, X_k))\\
&       + \sum_{0\leq i<j \leq k}(-1)^{i-1}c^k(X_0, \ldots, \widehat{X}_i, \ldots, X_{j-1}, [X_i, X_j], X_{j+1}, \ldots X_k),
\end{align*}
where $X_0, \ldots, X_k \in \Gamma A.$ Then $\{C^k(\Gamma A; C^\infty (M)), d_{\mathcal A}\}_{k\geq 0}$ is a cochain complex and the corresponding homology at the $k^{th}$-level is called the $k^{th}$ cohomology of $\mathcal A$ with coefficients in the trivial representation, and is denoted by $\mathcal H^k_{Leib}(\mathcal A).$
\end{defn}
It follows, in particular, that a $1$-cochain $c \in C^1(\Gamma A; C^\infty(M))$ is a $1$-cocycle provided
$$ \rho(X) (c(Y)) - \rho(Y)(c(X)) = c([X, Y]), ~~X, Y \in \Gamma A.$$

\section{Lie algebroid with Nambu structure}\label{3}
In this section, we recall the notion of a Lie algebroid with Nambu structure \cite{hagi} and discuss properties and examples of such structures.

\begin{defn}\label{Def}
Let $(A, [~,~],\rho)$ be a Lie algebroid. Then an $n$-multisection ($3\leq n \leq~\mbox{rank}~A$) $\Pi \in \Gamma (\wedge^nA)$ is called a {\it Nambu structure of order $n$} on $A$ if 
\begin{align}\label{lie-nambu-eqn}
 \mathcal{L}_{\Pi^{\sharp} \alpha} \Pi = (-1)^n (\iota_{d_A \alpha} \Pi) \Pi 
\end{align}
holds for all $\alpha \in \Gamma (\wedge^{n-1}A^*),$  where $\Pi^{\sharp}: \Gamma (\wedge^{n-1}A^*) \rightarrow \Gamma A$ is the induced bundle map $\alpha \mapsto \Pi^\sharp(\alpha) :=\iota_{\alpha}\Pi.$ A Lie algebroid $A$ with a Nambu structure $\Pi$ is denoted by the pair $(A, \Pi).$ A Nambu structure $\Pi$ on $A$ is called {\it maximal} if the order of $\Pi$ is equal to the rank of $A$.
\end{defn}

\begin{exam}\label{lie-nam-exam}
\begin{enumerate}
\item A Nambu-Poisson manifold $(M, \{ ~, \ldots, ~\})$ of order $n$ ($n \geq 3$) with associated Nambu-Poisson tensor $\Lambda \in  \Gamma(\wedge^n TM)$ induces a Nambu structure on the Lie algebroid $TM$. To see this, note that if $\Lambda(x) = 0$, for some $x \in M$, then both sides of Equation
(\ref{lie-nambu-eqn}) become zero at $x$. If $\Lambda (x) \neq 0$, then by Theorem \ref{local-structure-NP-manifold}, there exists a local chart $(U; x_1, \ldots, x_n, x_{n+1}, \ldots, x_m)$ of $M$ around $x$, $m$ being the dimension of $M$, such that on $U,$ $\Lambda$ is of the form $$\Lambda|U = \frac{\partial}{\partial x_1} \wedge \cdots \wedge \frac{\partial}{\partial x_n}.$$ Thus to verify the defining condition (\ref{lie-nambu-eqn}) of Definition \ref{Def}, it is sufficient to check it for local $(n-1)$-forms of the type 
$$\alpha = \sum_{i=1}^nf_idx_1 \wedge \cdots \wedge \widehat{dx_i} \wedge \cdots \wedge dx_n,~~f_i \in C^\infty(U).$$ In this case, we have 
$\Lambda^\sharp\alpha = \iota_{\alpha}\Lambda = \sum_{i=1}^n (-1)^{n-i}f_i\frac{\partial}{\partial x_i}.$ Therefore,
\begin{align*}
\mathcal L_{\Lambda^\sharp \alpha}\Lambda & = \sum_{i= 1}^n (-1)^{n-i}\mathcal L_{f_i\frac{\partial}{\partial x_i}}\frac{\partial}{\partial x_1} \wedge \cdots \wedge \frac{\partial}{\partial x_n}\\
& = \sum_{i= 1}^n (-1)^{n-i}(-1)^n(\iota_{df_i}\frac{\partial}{\partial x_1} \wedge \cdots \wedge \frac{\partial}{\partial x_n}) \wedge \frac{\partial}{\partial x_i}\\
& = \sum_{i= 1}^n (-1)^{n-i}(-1)^n (-1)^{i-1}(\iota_{df_i}\frac{\partial}{\partial x_i}\wedge \frac{\partial}{\partial x_1} \wedge \cdots \wedge \widehat{\frac{\partial}{\partial x_i}}\wedge \cdots \wedge \frac{\partial}{\partial x_n})\wedge \frac{\partial}{\partial x_i}\\
& = \sum_{i= 1}^n (-\frac{\partial f_i}{\partial x_i})\frac{\partial}{\partial x_1} \wedge \cdots \wedge \widehat{\frac{\partial}{\partial x_i}}\wedge \cdots \wedge \frac{\partial}{\partial x_n}\wedge \frac{\partial}{\partial x_i}\\
& = \sum_{i= 1}^n-(-1)^{n-i}\frac{\partial f_i}{\partial x_i}\frac{\partial}{\partial x_1} \wedge \cdots \wedge \frac{\partial}{\partial x_n}.
\end{align*}

On the other hand,
\begin{align*}
d\alpha & = \sum_{i=1}^ndf_i \wedge dx_1 \wedge \cdots \wedge \widehat{dx_i}\wedge \cdots \wedge dx_n\\
& = \sum_{i=1}^n(-1)^{i-1}dx_1 \wedge \cdots \wedge dx_{i-1}\wedge df_i \wedge dx_{i+1} \wedge  \cdots \wedge dx_n.
\end{align*}
Hence, $\iota_{d\alpha}\Lambda = \sum_{i=1}^n(-1)^{i-1}\frac{\partial f_i}{\partial x_i}.$ Therefore, we obtain $\mathcal L_{\Lambda^\sharp \alpha}\Lambda = (-1)^n (\iota_{d\alpha}\Lambda)\Lambda,$ proving that $\Lambda$ is a Nambu structure on the tangent Lie algebroid $TM$.

Conversely, any $n$-vector field $\Lambda \in  \Gamma(\wedge^n TM)$ which satisfies Equation (\ref{lie-nambu-eqn}) is certainly a Nambu-Poisson tensor. This follows by taking $\alpha = df_1 \wedge \cdots \wedge df_{n-1}$ in Equation (\ref{lie-nambu-eqn}) and comparing it with Equation (\ref{lie-der-nambu}). Thus Lie algebroid with Nambu
structure is a natural generalization of a Nambu-Poisson manifold of order $n$ ($n\geq 3$).

\item Let $(\mathfrak g, [~, ~])$ be a Lie algebra of dimension $m.$ Consider $\mathfrak g$ as a Lie algebroid over a point. Let $X_1, X_2, \ldots, X_n \in \mathfrak g$ be such that $[X_i,X_j] = 0,$ for all $i, j = 1, \ldots, n,$ $n \leq m.$ Then $\Pi = X_1 \wedge \cdots \wedge X_n \in \wedge^n \mathfrak g$ is a Nambu structure of order $n$ on the Lie algebra $\mathfrak g.$ For, if $X_1, X_2, \ldots, X_n $ are linearly dependent, then $\Pi$ is zero and therefore condition (\ref{lie-nambu-eqn}) of Definition \ref{Def} is trivially satisfied. In case $\{X_1, X_2, \ldots, X_n\}$ is linearly independent, we extend it to a basis $\{X_1, X_2, \ldots, X_n, X_{n+1}, \ldots X_m\}$ of $\mathfrak g.$ Let $\{X^*_1, X^*_2, \ldots, X^*_m\}$ be the dual basis of $\mathfrak g^*.$ To check that $\Pi$ is a Nambu structure it is enough to verify condition (\ref{lie-nambu-eqn}) for $\alpha \in \wedge^{n-1}\mathfrak g ^*$ of the forms 
\begin{itemize}
\item  [(i)] $\alpha = X^*_1 \wedge \cdots \wedge \widehat{X^*_k} \wedge \cdots \wedge X^*_n, ~~1\leq k \leq n$ and 
\item  [(ii)] $\alpha = X^*_{i_1} \wedge \cdots \wedge X^*_{i_{n-1}}, ~~i_k\notin \{1, \ldots, n\},$ for some $1\leq k \leq n-1.$
\end{itemize}
Note that both sides of the equality (\ref{lie-nambu-eqn}) are zero in these cases. This is because for $\alpha$ of the form (i),
\begin{align*} 
\mathcal{L}_{\Pi^{\sharp} \alpha} \Pi  &= [\Pi^{\sharp} \alpha, \Pi] = (-1)^{n-k}[X_k, X_1 \wedge \cdots \wedge X_n] \\
& = (-1)^{n-k}\sum_{i=1}^nX_1 \wedge \cdots \wedge[X_k, X_i]\wedge \cdots \wedge X_n =0,
\end{align*}
whereas, for $\alpha$ of the form (ii), $\mathcal{L}_{\Pi^{\sharp} \alpha} \Pi =0$ as $\Pi^{\sharp} \alpha=0.$ On the other hand, suppose for any $\alpha \in \wedge^{n-1}\mathfrak g^*$
$$\delta \alpha = cX^*_1\wedge \cdots \wedge X^*_n + \sum c_{i_1\ldots i_n}X^*_{i_1} \wedge \cdots \wedge X^*_{i_n},$$ where on the right hand side of the above equality the sum is over indices $i_1, \ldots i_n$ such that $i_1<\cdots <i_n,~~(i_1, \ldots, i_n) \neq (1, \ldots , n)$ and $\delta$ is the differential operator of the Chevalley-Eilenberg complex of $\mathfrak g.$ 
Then $c= (\delta\alpha)(X_1, \ldots, X_n) =0.$ For,
$$(\delta\alpha)(X_1, \ldots, X_n) = \sum_{1\leq i<j \leq n}(-1)^{i+j} \alpha([X_i, Xj], X_1, \ldots, \widehat{X_i}, \ldots, \widehat{X_j}, \ldots, X_n) =0,$$ as $[X_i, X_j]=0,~~i, j \in \{1, \ldots, n\}.$ Thus $\iota_{\delta\alpha}\Pi= 0$ for any $\alpha \in \wedge^{n-1}\mathfrak g^*.$ Hence, $$(-1)^n( \iota_{\delta\alpha}\Pi)\Pi = 0.$$

Suppose $\mathfrak g$ is the Lie algebra of a Lie group $G.$ Let $\overleftarrow{X}_1, \ldots, \overleftarrow{X}_n$ be the left invariant vector fields corresponding to $X_1, X_2, \ldots, X_n \in \mathfrak g.$ Then, $[\overleftarrow{X}_i, \overleftarrow{X}_j] = 0,$ for all $i, j = 1, \ldots, n$ and therefore, $\overleftarrow{\Pi} = \overleftarrow{X}_1 \wedge \cdots \wedge \overleftarrow{X}_n$ is the left invariant Nambu-Poisson structure on $G$ such that $\overleftarrow{\Pi}(e) = \Pi.$  
\end{enumerate}
\end{exam}

The following result shows that for an orientable Lie algebroid each choice of an orientation form gives rise to a Nambu structure. In particular, this result gives rise to examples of Lie algebroids equipped with Nambu structure which do not come from Nambu-Poisson manifolds (cf. Example \ref{Nambu-structure-not-NP}).

\begin{prop}\label{volume-nambu}
Let $A$ be an oriented Lie algebroid of rank $m$ ($m \geq 3$) and $\mu \in \Gamma(\wedge^m A^*)$ be a non-vanishing section representing the orientation of $A$.
Define an $m$-multisection $\Pi_{\mu} \in \Gamma(\wedge^m A)$ by $\Pi_{\mu} (\alpha_1, \ldots, \alpha_m) \mu = \alpha_1 \wedge \cdots \wedge \alpha_m,$
for all $\alpha_1, \ldots, \alpha_m \in \Gamma A^*.$ Then $\Pi_\mu$ defines a maximal Nambu structure on the Lie algebroid $A.$ 
\end{prop}

\begin{proof}
We first prove that the evaluation function $\langle \Pi_\mu, \mu \rangle$ is identically $1$ on $M$. Let $x \in M$ and $\alpha_1, \ldots, \alpha_m$ be local sections
of the bundle $A^*$ around $x$ such that
\begin{align*}
 \mu (x) = \alpha_1 (x) \wedge \cdots \wedge \alpha_m (x).
\end{align*}
Then from the definition of $\Pi_\mu$ we get 
$$\langle \Pi_\mu (x), \mu (x) \rangle \mu (x) = \langle \Pi_\mu (x), \alpha_1 (x) \wedge \cdots \wedge \alpha_m (x) \rangle \mu (x) = \alpha_1 (x) \wedge \cdots \wedge \alpha_m (x) = \mu (x) .$$ Since $\mu (x) \neq 0$, we have
$\langle \Pi_\mu (x), \mu (x) \rangle =1$, which implies $\langle \Pi_\mu, \mu \rangle = 1$.

Hence we have $\partial_\mu(\Pi_\mu) = *_\mu^{-1} \circ d_A \circ *_\mu (\Pi_\mu) = 0$. Therefore, from Proposition \ref{formula}, we get
\begin{align*}
 \text{div}_\mu (\Pi_\mu^{\sharp} \alpha) + (-1)^m \iota_{d_A \alpha} \Pi_\mu = 0,
\end{align*}
for any $\alpha \in \Gamma (\wedge^{m-1}A^\ast).$
In other words,
\begin{align}\label{eqn1}
 \mathcal{L}_{\Pi_\mu^{\sharp} \alpha} \mu = (-1)^{m-1} (\iota_{d_A \alpha} \Pi_\mu) \mu.
\end{align}
From property (8) of Proposition \ref{properties-Lie derivative-contraction} we get
$$ 0 = \mathcal{L}_{\Pi_\mu^{\sharp} \alpha} \langle \Pi_\mu, \mu \rangle = \mathcal L_{\Pi_\mu^\sharp\alpha}\iota_{\Pi_\mu}\mu = \iota_{\Pi_\mu}\mathcal L_{\Pi_\mu^\sharp\alpha}\mu + \iota_{[\Pi_\mu^\sharp\alpha, \Pi_\mu]}\mu
= \langle \Pi_\mu, \mathcal{L}_{\Pi_\mu^{\sharp} \alpha}  \mu \rangle + \langle \mathcal{L}_{\Pi_\mu^{\sharp} \alpha} \Pi_\mu , \mu \rangle.$$
Therefore, using Equation (\ref{eqn1}), we get
\begin{align*}
 \langle \mathcal{L}_{\Pi_\mu^{\sharp} \alpha} \Pi_\mu , \mu \rangle = (-1)^m (\iota_{d_A \alpha} \Pi_\mu) \langle \Pi_\mu, \mu \rangle = (-1)^m (\iota_{d_A \alpha} \Pi_\mu).
\end{align*}
Note that $\Pi_\mu \in \Gamma(\wedge^\text{top}A)$ is nowhere vanishing, hence, $\mathcal{L}_{\Pi_\mu^{\sharp} \alpha} \Pi_\mu = (-1)^m (\iota_{d_A \alpha} \Pi_\mu) \Pi_\mu.$
Thus, $\Pi_\mu$ is a Nambu-structure of order $m$ and hence a maximal Nambu structure.
\end{proof}

\begin{remark}\label{f-nambu}
Let $\Pi \in \Gamma(\wedge^{\text{top}} A)$ be an arbitrary multisection of $A$. Then there exists a function $f \in C^\infty(M)$ such that
$\Pi = f \Pi_\mu$. Observe that
\begin{align*}
\mathcal{L}_{\Pi^{\sharp} \alpha} \Pi & = \mathcal{L}_{f \Pi_\mu^{\sharp} \alpha} f \Pi_\mu = f \mathcal L_{\Pi_\mu^{\sharp} \alpha}f\Pi_\mu - (\rho(\Pi_\mu^{\sharp} \alpha)f)f\Pi_\mu ~~ \mbox{(by Remark \ref{properties-top-degree} (1))}\\
& = f^2 \mathcal{L}_{\Pi_\mu^{\sharp} \alpha} \Pi_\mu - f ( \rho \Pi_\mu^{\sharp}(\alpha)f ) \Pi_\mu +  f ( \rho \Pi_\mu^{\sharp}(\alpha)f ) \Pi_\mu~~ \mbox{(by Remark \ref{properties-top-degree} (2))}\\
& = f^2 \mathcal{L}_{\Pi_\mu^{\sharp} \alpha} \Pi_\mu.
\end{align*}
Thus, $\mathcal{L}_{\Pi^{\sharp} \alpha} \Pi = f^2 (-1)^m (\iota_{d_A \alpha} \Pi_\mu) \Pi_\mu = (-1)^m (\iota_{d_A\alpha} \Pi) \Pi.$

Hence, $\Pi$ defines a (maximal) Nambu structure of order $m$. In other words, any $m$-multisection defines a Nambu structure.
We write $\Pi = \Pi_f$ to understand its relation with $\Pi_\mu.$
\end{remark}

\begin{exam}\label{Nambu-structure-not-NP}
Let $S^3$ be the unit sphere in $\mathbb R^4$ with coordinates $(a, b, x, y),~~ a^2 + b^2 + x^2 + y^2 = 1$. It is known that \cite{grabowski} $S^3$ admits a Poisson structure $\pi = A\wedge B,$ where
$$A = aC-\frac{\partial}{\partial a},~~B = -bC + \frac{\partial}{\partial b},~~\text{and}~~C = a\frac{\partial}{\partial a} + b\frac{\partial}{\partial b} + x\frac{\partial}{\partial x} + y\frac{\partial}{\partial y}.$$
If we identify this sphere with the group $SU(2)$ by
$$(a, b, x, y)\mapsto \left( \begin{array}{cc}
\alpha & -\bar{\nu}\\ \nu & \bar{\alpha}
\end{array}
\right), $$ where $\alpha = a+ ib,~~ \nu = x+iy,$ then $\pi$ can also be written in the form
$$\pi = (x^2 + y^2) \hat{X} \wedge \hat{Y} + (by +ax) \hat{X}\wedge \hat{H} + (ay - bx) \hat{Y} \wedge \hat{H},$$
where $\hat{X},~~\hat{Y},~~\hat{H}$ are the right-invariant vector fields on $SU(2)$ corresponding to the elements
$$X = \left( \begin{array}{cc}
0 & -1\\ 1 & 0
\end{array}
\right), ~~Y = \left( \begin{array}{cc}
0 & i\\ i & 0
\end{array}
\right), ~~H = \left( \begin{array}{cc}
i & 0\\ 0 & i
\end{array}
\right)$$ of the Lie algebra $su(2)$ of $SU(2).$ Thus the cotangent bundle $T^\ast SU(2)$ is an orientable Lie algebroid and hence admits a maximal Nambu structure by Proposition \ref{volume-nambu}. More generally, for any Poisson Lie group $G,$ of dimension $n \geq 3,$ its cotangent bundle $T^\ast G$ is an orientable Lie algebroid and hence admits a maximal Nambu structure.  
\end{exam}

\begin{exam}
Suppose $\phi : \mathfrak g \rightarrow \Gamma TM$ is a smooth action of a finite dimensional real Lie algebra $(\mathfrak g, [~, ~])$ on a smooth manifold $M$. Thus $\phi$ is a Lie algebra homomorphism from $\mathfrak g$ to the Lie algebra of vector fields on $M$. Let $A = M \times \mathfrak g \rightarrow M$ be the trivial bundle with fibre $\mathfrak g$. Then $A$ is a Lie algebroid, called the action Lie algebroid, whose bracket and anchor are given by
$$[u, v](x) = [u(x), v(x)] + (\phi(u(x))v)(x) - (\phi(v(x))u)(x); ~~\rho (x, X) = \phi(X)(x),$$
for  $u, v \in C^\infty(M , \mathfrak g),$ $x \in M$ and $X\in \mathfrak g$ \cite{mac-book}. If  dimension of $\mathfrak g$ is $m,$ then $A$ is a trivial vector bundle of rank $m$ and $\wedge^m A = M \times \wedge^m \mathfrak g = M \times \mathbb R.$ It follows from Remark \ref{f-nambu} that any smooth function on $M$ can be considered as a Nambu structure of order $m$ on the action Lie algebroid $A$.
\end{exam}

Given a Nambu-Poisson manifold $M$ of order $n$, one can define a family of Nambu-Poisson structures on $M$ of lower order than $n$, called subordinate structures \cite {takhtajan}, which satisfy certain matching conditions. In the following example we prove that a similar result holds in the general context. More specifically, we prove that a given  Nambu structure on a Lie algebroid $A$ induces subordinate structures on $A$ of lower order. 

\begin{exam}(Subordinate Nambu structures) \label{sub}
Let $A$ be a Lie algebroid with a Nambu structure $\Pi \in \Gamma(\wedge^nA)$ of order $n$ $(n > 3)$. Let $k$ be a positive integer such that $ n-k \geq 3.$ Let $\alpha_1, \ldots, \alpha_k \in \Gamma{A^*}$ be such that $d_A \alpha_i = 0$, for all $i=1, \ldots, k$. Consider the
$(n-k)$- multisection $\widetilde{\Pi} \in \Gamma(\wedge^{n-k}A)$ defined by $\widetilde{\Pi} = \iota_{\bar{\alpha}} \Pi,$ where $\bar{\alpha} = \alpha_1 \wedge \cdots \wedge \alpha_k$. Then the induced bundle map ${\widetilde{\Pi}}^{\sharp} : \Gamma(\wedge^{n-k-1}A^*) \rightarrow \Gamma A$ is given by ${\widetilde{\Pi}}^{\sharp} (\beta) = \iota_{\bar{\alpha} \wedge \beta} \Pi = \Pi^{\sharp}(\bar{\alpha} \wedge \beta),$ for all $\beta \in \Gamma(\wedge^{n-k-1}A^*).$ Therefore,
\begin{align*}
\mathcal{L}_{{\widetilde{\Pi}}^{\sharp} (\beta)} \widetilde{\Pi} & = \mathcal{L}_{ \Pi^{\sharp}(\bar{\alpha} \wedge \beta) } \iota_{\bar{\alpha}} \Pi = \iota_{\bar{\alpha}} \mathcal{L}_{ \Pi^{\sharp}(\bar{\alpha} \wedge \beta) } \Pi + \iota_{(\mathcal{L}_{ \Pi^{\sharp}(\bar{\alpha} \wedge \beta) } \bar{\alpha})} \Pi \\
& = (-1)^n \iota_{\bar{\alpha}} \big( \iota_{d_A(\bar{\alpha} \wedge \beta)}   \Pi   \big) \Pi +  \iota_{(\mathcal{L}_{ \Pi^{\sharp}(\bar{\alpha} \wedge \beta) } \bar{\alpha})} \Pi \\
& = (-1)^n \big( \iota_{d_A{\bar{\alpha}} \wedge \beta} \Pi \big) \iota_{\bar{\alpha}} \Pi + (-1)^{n-k} \big( \iota_{\bar{\alpha} \wedge d_A \beta} \Pi \big) \iota_{\bar{\alpha}} \Pi
+ \iota_{(\mathcal{L}_{ \Pi^{\sharp}(\bar{\alpha} \wedge \beta) } \bar{\alpha})} \Pi.
\end{align*}
Note that $\bar{\alpha} = \alpha_1 \wedge \cdots \wedge \alpha_k$ is  $d_A$-closed as each $\alpha_i$ is $d_A$-closed. Therefore, the first term on  the 
right hand side of the above equality vanishes, while the last term also vanishes by using the Cartan formula and the property of the contraction operator. Hence,
$$ \mathcal{L}_{{\widetilde{\Pi}}^{\sharp} (\beta)} \widetilde{\Pi} = (-1)^{n-k} \big( \iota_{\bar{\alpha} \wedge d_A \beta} \Pi \big) \iota_{\bar{\alpha}} \Pi = (-1)^{n-k} \big( \iota_{d_A \beta} \iota_{\bar{\alpha}} \Pi  \big) \iota_{\bar{\alpha}} \Pi = (-1)^{n-k} ( \iota_{d_A \beta} \widetilde{\Pi} ) \widetilde{\Pi}.$$

Thus, $\widetilde{\Pi}$ is a Nambu structure of order $n-k \geq 3$. We call these Nambu structures as {\it subordinate} Nambu structures of $\Pi$.
\end{exam}

\begin{defn}\label{hamiltonian-section}
Let $(A, \Pi)$ be a Lie algebroid with a Nambu structure of order $n$. Let $f_1, \ldots, f_{n-1} \in C^\infty(M).$ The section $\Pi^\sharp( d_Af_1\wedge \cdots \wedge d_Af_{n-1}) \in \Gamma A$ is called the Hamiltonian $A$-section corresponding to the Hamiltonians $f_1, \ldots, f_{n-1} \in C^\infty(M)$.
\end{defn}

Observe that for a Nambu-Poisson manifold $M,$ if we consider the Example \ref{lie-nam-exam} (1) then the Hamiltonian $TM$-section corresponding to the Hamiltonians $f_1, \ldots, f_{n-1}$ is precisely the Hamiltonian vector field $X_{f_1 \ldots f_{n-1}}.$
\begin{remark}\label{equivalence-of-defn}
It may be remarked that A. Wade \cite{wade} introduced the notion of a Nambu structure (of order $n$) on a Lie algebroid $A$ as an $n$-multisection $\Pi \in \Gamma(\wedge^nA)$ such that 
$$[ \Pi^\sharp \alpha, \Pi]^\sharp(\beta) = - \Pi^\sharp(\iota_{\Pi^\sharp\beta}d_A\alpha),$$ for all $\alpha, \beta \in \Gamma(\wedge^{n-1}A^\ast),$ where the bracket on the left hand side of the above equality is the Schouten bracket of multisections of $A$.
In the following proposition we show that Definition \ref{Def} of a Nambu structure on a Lie algebroid  and that of \cite{wade} are equivalent if the multisection $\Pi \in \Gamma(\wedge^nA)$ $(n\geq 3)$ is locally decomposable.
\end{remark}

\begin{defn}\label{locally-decomposable}
An $n$-multisection $\Pi \in \Gamma(\wedge^nA)$ $(n \geq 2)$ is called locally decomposable if  at each point $x \in M$, either $\Pi(x) = 0$ or there exists a neighborhood $U$ of $x$ such that
\begin{align*}
 \Pi|_U = X_1 \wedge \cdots \wedge X_n,
\end{align*}
for some local sections $X_i$ of the bundle $A$ defined on $U$.
\end{defn}

\begin{prop}\label{proof-equivalence-of-defn}
Let $(A, [~, ~], \rho)$ be a Lie algebroid over $M$. If $\Pi \in \Gamma (\wedge^nA)$ is a locally decomposable $n$-multisection of $A,~~n\geq 2,$ then the following conditions are equivalent.
\begin{enumerate}
\item $\mathcal L _{\Pi^\sharp\alpha}\Pi = (-1)^n(\iota_{d_A\alpha}\Pi)\Pi,$ for all $\alpha \in \Gamma (\wedge^{n-1}A^\ast).$
\item $[\Pi^\sharp \alpha, \Pi]^\sharp(\beta) = - \Pi^\sharp(\iota_{\Pi^\sharp\beta}d_A\alpha),$ for all $\alpha,~ \beta \in \Gamma(\wedge^{n-1}A^\ast).$
\end{enumerate}
\end{prop}

\begin{proof}
It is enough to show that
\begin{align}\label{proof-equiv-defn}
- \Pi^\sharp(\iota_{\Pi^\sharp\beta}d_A\alpha) = (-1)^n(\iota_{d_A\alpha}\Pi)\Pi^\sharp\beta, \hspace{2mm}\alpha,~ \beta \in \Gamma(\wedge^{n-1}A^\ast).
\end{align}
If $\Pi (x)= 0 $ for some $x \in M,$ then both sides of the above equality are zero at $x \in M.$ Suppose $\Pi (x) \neq 0$ for $x \in M$. Then there exists an open neighbourhood $U$ of $x$ in $M$ and local sections $X_1, \ldots, X_n \in \Gamma(A|_U)$ such that $\Pi$ is non-vanishing on $U$ and $$\Pi |_U = X_1 \wedge \cdots \wedge X_n.$$ We may assume that $U$ is a trivializing neighbourhood of the vector bundle $A$. Let the rank of $A$ be $m$. We extend $\{X_1, \ldots , X_n\}$ to  a trivialization $\{X_1, \ldots , X_n, X_{n+1}, \ldots , X_m\}$ of $A|_U.$ Consider the dual basis $\{X^\ast_1, \ldots , X^\ast_n, X^\ast_{n+1}, \ldots , X^\ast_m\}$ of sections of $A^\ast|_U.$ To prove the equality (\ref{proof-equiv-defn}), it is enough to consider the case for which
$$(d_A\alpha)|_U = a X^\ast_1\wedge \cdots \wedge X^\ast_n + \alpha^\prime, ~~\beta|_U = b X^\ast_1\wedge \cdots \wedge \widehat{X^\ast_k}\wedge \cdots \wedge X^\ast_n,~~1\leq k\leq n,$$
where $a,~ b \in C^\infty(U)$ and $\alpha ^\prime \in \Gamma(\wedge^nA^\ast|_U)$ involves terms $X^\ast_{i_1}\wedge \cdots \wedge X^\ast_{i_n}$ with $\{i_1, \ldots ,i_n\} \neq \{1, \ldots , n\}.$ In the above case, the left hand side of the equality (\ref{proof-equiv-defn}) becomes
\begin{align*}
-\Pi^\sharp(\iota_{\Pi^\sharp\beta}d_A\alpha) & = -(-1)^{n-k}ab\Pi^\sharp (\iota_{X_k}X^\ast_1\wedge \cdots \wedge X^\ast_n)\\
& = -(-1)^{n-k}(-1)^{k-1}ab\Pi^\sharp (X^\ast_1\wedge \cdots \wedge \widehat{X^\ast_k}\wedge \cdots \wedge X^\ast_n)\\
& = -(-1)^{n-k}(-1)^{k-1}(-1)^{n-k}abX_k = (-1)^kab X_k.  
\end{align*}
On the other hand, the right hand side of the equality (\ref{proof-equiv-defn}) becomes
$$(-1)^n(\iota_{d_A\alpha}\Pi)\Pi^\sharp\beta =(-1)^na\Pi^\sharp\beta = (-1)^n (-1)^{n-k}ab X_k = (-1)^kab X_k.$$ Thus the required equality holds. 
\end{proof}

\begin{remark}
\begin{enumerate}
\item At any point $x \in M$ the equality (\ref{proof-equiv-defn}) depends only on the values of $\Pi$, $\beta$ and $d_A\alpha$ at $x.$ Thus if $\Pi$ is decomposable at a point $x$, then the statement (1) of the above proposition holds at $x$ if and only if the statement (2) holds at $x.$
\item Although, in the present paper, we are concerned with Nambu structures of order $n >2,$ the defining condition (\ref{lie-nambu-eqn}) of a Nambu structure makes sense for $n=2$. However, as remarked in the introduction and also from the above proposition, it is clear that a Poisson bivector field will satisfy (\ref{lie-nambu-eqn}) for $n =2$ provided it is locally decomposable. For illustration, let us consider the Poisson structure 
$$\pi = \frac{\partial}{\partial x_1} \wedge \frac{\partial}{\partial y_1} + \frac{\partial}{\partial x_2} \wedge \frac{\partial}{\partial y_2}$$ on $\mathbb R^4$ which corresponds to the canonical symplectic structure on $\mathbb R^4.$  We claim that $\pi$ does not satisfy (\ref{lie-nambu-eqn}) for $n =2.$ To see this, observe that for $ \alpha = y_1dx_1,$ 
$$\mathcal L_{\pi^\sharp\alpha}\pi = \mathcal L_{y_1\frac{\partial}{\partial y_1}}\pi = - \frac{\partial}{\partial x_1} \wedge \frac{\partial}{\partial y_1}, ~~(\iota_{d\alpha}\pi)\pi = -\pi = -\frac{\partial}{\partial x_1} \wedge \frac{\partial}{\partial y_1} - \frac{\partial}{\partial x_2} \wedge \frac{\partial}{\partial y_2}.$$
Therefore,   $\mathcal L_{\pi^\sharp\alpha}\pi\neq (-1)^2(\iota_{d\alpha}\pi)\pi = (\iota_{d\alpha}\pi)\pi.$
\end{enumerate}
\end{remark}

It is known that the Nambu tensor of a Nambu-Poisson manifold of order $> 2$ is locally decomposable (cf. Theorem \ref{local-structure-NP-manifold}). The following result (Lemma 3.9, \cite{hagi}) provides a sufficient condition for a Nambu structure on a Lie algebroid to be locally decomposable. For completeness, we give a detailed proof. 

\begin{prop}\label{condition-for-local-decomposability}
 Let $(A, \Pi)$ be a Lie algebroid with a Nambu structure $\Pi \in \Gamma(\wedge^nA)$ of order $n \geq 3$. If $\Gamma{A^*}$ is locally generated by elements of the form
$d_Af,~~f \in C^\infty(M)$, then the Nambu tensor $\Pi$ is locally decomposable.
\end{prop}

\begin{proof}
Let $f \in C^\infty(M)$ and $\beta \in \Gamma(\wedge^{n-2}A^*)$ be arbitrary. Take $\alpha = f d_Af \wedge \beta \in \Gamma(\wedge^{n-1}A^*)$. Note that by Proposition \ref{properties-Lie derivative-contraction} we have
\begin{align*}
 \mathcal{L}_{\Pi^{\sharp} \alpha} \Pi =& \mathcal{L}_{ f \Pi^{\sharp} (d_Af \wedge \beta)} \Pi = f \mathcal{L}_{ \Pi^{\sharp} (d_Af \wedge \beta)} \Pi - \Pi^{\sharp} (d_Af \wedge \beta) \wedge (\iota_{d_Af} \Pi) \\
=& (-1)^n f \big( \iota_{d_A (d_Af \wedge \beta) } \Pi \big) \Pi - \iota_{d_Af \wedge \beta} \Pi \wedge (\iota_{d_Af} \Pi)\\ 
=& -(-1)^n f  (\iota_{d_Af \wedge d_A\beta} \Pi) \Pi - (\iota_\beta \iota_{d_Af} \Pi) \wedge (\iota_{d_Af} \Pi)
\end{align*}
and
\begin{align*}
 (-1)^n \big( \iota_{d_A \alpha} \Pi \big) \Pi = -(-1)^n \big( \iota_{f d_Af \wedge d_A\beta} \Pi \big) \Pi = -(-1)^n f \big( \iota_{ d_Af \wedge d_A\beta} \Pi \big) \Pi.
\end{align*}
Since $\Pi$ is a Nambu structure of order $n$, we have $\mathcal{L}_{\Pi^{\sharp} \alpha} \Pi = (-1)^n \big( \iota_{d_A \alpha} \Pi \big) \Pi.$
Therefore, $(\iota_\beta \iota_{d_Af} \Pi) \wedge (\iota_{d_Af} \Pi) = 0,$
for all $f \in C^\infty(M)$ and $\beta \in \Gamma(\wedge^{n-2}A^*)$. This implies that $\iota_{d_Af} \Pi$ is locally decomposable. If
$\Gamma{A^*}$ is locally generated by elements of the form $d_Af$, then the above identity implies $\Pi^{\sharp}(\alpha) \wedge (\iota_\eta \Pi) = 0,$
for all $\alpha \in \Gamma(\wedge^{n-1}A^*)$ and $\eta \in \Gamma{A^*},$ proving local decomposability of $\Pi.$
\end{proof}

\begin{remark}
\begin{enumerate}
\item Any regular foliation $\mathcal{F}$ can be considered as a Lie algebroid with the inclusion map as its anchor. Then any Nambu structure on $\mathcal{F}$
is locally decomposable.
\item It may be remarked that A. Wade also proved a version of the above result (Proposition \ref{condition-for-local-decomposability}) (cf. Theorem 3.4 \cite{wade}). Thus it follows from Proposition \ref{proof-equivalence-of-defn} that for a Lie algebroid $A$ for which $\Gamma A^\ast$ is locally generated by elements of the form $d_Af,~~f \in C^\infty(M)$ the two definitions of Nambu structures are equivalent.  
\end{enumerate}
\end{remark}

In the next proposition, we show that if we have a Nambu structure on a Lie algebroid then the base manifold carries a natural Nambu-Poisson structure.

\begin{prop}\label{induced-NP-structure-base-prop}
 Let $(A, \Pi)$ be a Lie algebroid over $M$ with a Nambu structure $\Pi \in \Gamma (\wedge^nA)$ of order $n$, $n \geq 3$. Then the base $M$ is equipped with an induced Nambu-Poisson
structure of order $n$.
\end{prop}

\begin{proof}\label{induced-NP-structure-base}
Define a $n$-bracket on $C^\infty(M)$ by
\begin{align*}
 \{f_1, \ldots, f_n\}_{\Pi} =  \Pi (d_A f_1, \ldots, d_A f_n),
\end{align*}
for $f_1, \ldots, f_n,  \in C^\infty (M)$. The bracket is obviously skew-symmetric and satisfies derivation property. Let $g \in C^\infty(M).$ Observe that
\begin{align*}
 \mathcal{L}_{\Pi^{\sharp}(d_A f_1 \wedge \cdots \wedge d_A f_{n-1})} g =& \rho \Pi^{\sharp}(d_A f_1 \wedge \cdots \wedge d_A f_{n-1}) (g) = \langle \Pi^{\sharp}(d_A f_1 \wedge \cdots \wedge d_A f_{n-1}), d_A g \rangle\\
=& \Pi (d_A f_1, \ldots , d_A f_{n-1}, d_A g) = \{f_1, \ldots, f_{n-1}, g \}_{\Pi}.
\end{align*}
Therefore, for $g_1, \ldots , g_n \in C^\infty(M),$
\begin{align*}
&\{f_1, \ldots, f_{n-1}, \{g_1, \ldots, g_n\}_\Pi \}_\Pi = \mathcal{L}_{\Pi^{\sharp}(d_A f_1 \wedge \cdots \wedge d_A f_{n-1})} \Pi (d_A g_1, \ldots, d_A g_n)\\
=& (\mathcal{L}_{\Pi^{\sharp}(d_A f_1 \wedge \cdots \wedge d_A f_{n-1})} \Pi) (d_A g_1, \ldots, d_A g_n) + \sum_{i=1}^n \Pi (d_A g_1, \ldots, \mathcal{L}_{\Pi^{\sharp}(d_A f_1 \wedge \cdots \wedge d_A f_{n-1})} d_A g_i, \ldots, d_A g_n)\\
=& \sum_{i=1}^n \Pi (d_A g_1, \ldots, d_A \{f_1 , \ldots , f_{n-1},  g_i\}_\Pi, \ldots, d_A g_n).
\end{align*}
Observe that the first term on the right hand side of the above equality vanishes since $\Pi$ is a Nambu structure. Thus,
$$\{f_1, \ldots, f_{n-1}, \{g_1, \ldots, g_n\}_\Pi \}_\Pi = \sum_{i=1}^n \{g_1, \ldots, \{f_1 , \ldots , f_{n-1},  g_i\}_\Pi, \ldots, g_n \}_\Pi,$$
hence the bracket $\{~, \ldots, ~\}_\Pi$ satisfies the fundamental identity, proving that $(M,\{~, \ldots, ~\}_{\Pi})$ is a Nambu-Poisson manifold.
\end{proof}

Next we prove that the converse of the above proposition holds under suitable condition on the Lie algebroid.
\begin{prop}
Suppose $(A, ~[~,~], \rho)$ is a Lie algebroid over a Nambu-Poisson manifold $M$ with a Nambu tensor $\Lambda \in \Gamma(\wedge^nTM).$ Assume that $\Gamma A^\ast$ is locally generated by elements of the form $d_Af = \rho^\ast df,$ for $f \in C^\infty(M).$ Then the $n$-multisection $\Pi \in \Gamma(\wedge^nA)$ defined by 
$$\Pi (d_Af_1, \ldots , d_Af_n) = \Lambda(df_1, \ldots ,df_n),$$ 
for $f_1, \ldots, f_n \in C^\infty (M),$ is a Nambu structure on the Lie algebroid $A.$ The induced Nambu structure on $M$ coincides with $\Lambda.$ 
\end{prop}

\begin{proof}
First observe that $\Gamma(\wedge^{n-1}A^\ast)$ is locally generated by elements of the form 
$$(\wedge^{n-1}\rho^\ast)(\alpha),~~\alpha \in \Gamma(\wedge^{n-1}T^\ast M).$$ To check that $\Pi$ is a Nambu structure, let $\alpha \in \Gamma(\wedge^{n-1}T^\ast M).$ Then by Definition \ref{Lie derivative-contraction} of the Lie derivative operator we have
\begin{align*}
& (\mathcal L_{\Pi^\sharp(\wedge^{n-1}\rho^\ast(\alpha))}\Pi)(d_Af_1, \ldots, d_Af_n)\\
& = \mathcal L_{\Pi^\sharp(\wedge^{n-1}\rho^\ast(\alpha))}(\Pi(d_Af_1, \ldots, d_Af_n)) - \sum_{i=1}^n \Pi(d_Af_1, \ldots, \mathcal L_{\Pi^\sharp(\wedge^{n-1}\rho^\ast(\alpha))}d_Af_i, \ldots , d_Af_n)\\
& = \mathcal L_{\Lambda^\sharp(\alpha)}\Lambda (df_1, \ldots, df_n) - \sum_{i=1}^n \Lambda(df_1, \ldots, \mathcal L_{\Lambda^\sharp(\alpha)}df_i, \ldots , df_n)\\
& = (\mathcal L_{\Lambda^\sharp(\alpha)}\Lambda)(df_1, \ldots, df_n)\\
& = (-1)^n (\iota_{d\alpha}\Lambda)\Lambda(df_1, \ldots, df_n)~~\mbox{(as $\Lambda$ is a Nambu structure on $TM.$)}\\
& = (-1)^n(\iota_{d_A(\wedge^{n-1}\rho^\ast(\alpha))}\Pi)\Pi(d_Af_1, \ldots, d_Af_n). 
\end{align*}

Thus $\Pi$ is a Nambu structure on $A$. The last statement is clear from the definition of $\Pi$.
\end{proof}

It is known that \cite{ibanez}, if $M$ is a Nambu-Poisson manifold of order $n$, then $\wedge^{n-1}T^*M$ carries a Leibniz algebroid structure. For Lie algebroid with Nambu structure we have a similar result as proved below.

\begin{prop}\label{leib-algbd}
 Let $(A, \Pi)$ be a Lie algebroid over $M$ with a Nambu structure of order $n$, $n \geq 3$. Then the triple
$(\wedge^{n-1}A^*, [ ~,~ ], \rho \circ \Pi^{\sharp})$ is a Leibniz algebroid over $M$, where the bracket $[~,~]$ is defined by
\begin{align}\label{leibniz-bracket}
 [ \alpha, \beta ] = \mathcal{L}_{\Pi^{\sharp} \alpha} \beta + (-1)^n (\iota_{d_A \alpha} \Pi) \beta,
\end{align}
for all $\alpha, \beta \in \Gamma(\wedge^{n-1}A^*).$
\end{prop}

\begin{proof}
 The bracket $[ ~, ~ ]$ defined above clearly satisfies derivation property
\begin{align*}
	[ \alpha, f \beta ] = f [ \alpha, \beta ] + (\rho \Pi^{\sharp}(\alpha) f ) \beta.
\end{align*}
The rest of the proof is divided into the following steps:\\
{\bf Step 1}: We claim that
\begin{align}
	[\Pi^{\sharp} \alpha, \Pi^{\sharp} \beta] = \Pi^{\sharp}  [ \alpha, \beta ].
\end{align}
This is because
\begin{align*}
  [\Pi^{\sharp} \alpha, \Pi^{\sharp} \beta] =& \iota_\beta [\Pi^{\sharp} \alpha, \Pi] + \iota_{ \mathcal{L}_{\Pi^{\sharp} \alpha} \beta} \Pi = \iota_\beta  \mathcal{L}_{\Pi^{\sharp} \alpha} \Pi +  \iota_{ \mathcal{L}_{\Pi^{\sharp} \alpha} \beta} \Pi \\
=& (-1)^n ( \iota_{d_A \alpha} \Pi) \iota_\beta \Pi + \iota_{ \mathcal{L}_{\Pi^{\sharp} \alpha} \beta} \Pi  = \iota_{[ \alpha,  \beta ]} \Pi = \Pi^{\sharp} [ \alpha,  \beta ].
\end{align*}
Moreover, it follows that $\rho \Pi^{\sharp} [ \alpha,  \beta ] =  [\rho \Pi^{\sharp} \alpha, \rho \Pi^{\sharp} \beta] $.\\
{\bf Step 2}: Next we prove that 
\begin{align} \label{eqn2}
 \iota_{d_A [ \alpha,  \beta ]} \Pi = \rho \Pi^{\sharp} (\alpha) \cdot (\iota_{d_A \beta} \Pi) - \rho \Pi^{\sharp} (\beta) \cdot (\iota_{d_A \alpha} \Pi).
\end{align}
To see this, note that if $\Pi (x) = 0$, for $x \in M$, then both sides of the equation is zero at $x$, since $(\rho \Pi^{\sharp}(\alpha))_x = 0$. Otherwise, observe that
\begin{align*}
& (-1)^n (\iota_{d_A [ \alpha,  \beta ]} \Pi) \Pi = \mathcal{L}_{\Pi^{\sharp}  [ \alpha, \beta ] } \Pi = \mathcal{L}_{[ \Pi^{\sharp} \alpha, \Pi^{\sharp} \beta ] } \Pi \hspace*{1cm} (\text{using}~~ {\bf Step ~~1}) \\
=& \mathcal{L}_{\Pi^{\sharp} \alpha} \mathcal{L}_{\Pi^{\sharp} \beta} \Pi - \mathcal{L}_{\Pi^{\sharp} \beta} \mathcal{L}_{\Pi^{\sharp} \alpha} \Pi = \mathcal{L}_{\Pi^{\sharp} \alpha} \big(  (-1)^n (\iota_{d_A \beta} \Pi) \Pi  \big) - \mathcal{L}_{\Pi^{\sharp} \beta} \big(  (-1)^n (\iota_{d_A \alpha} \Pi)  \Pi  \big) \\
=& (-1)^n \bigg[ (\iota_{d_A \beta} \Pi) \mathcal{L}_{\Pi^{\sharp} \alpha} \Pi + \rho \Pi^{\sharp} (\alpha) \cdot (\iota_{d_A \beta} \Pi) \Pi \bigg] \\
& \hspace*{0.5cm} - (-1)^n \bigg[ (\iota_{d_A \alpha} \Pi) \mathcal{L}_{\Pi^{\sharp} \beta} \Pi + \rho \Pi^{\sharp} (\beta) \cdot (\iota_{d_A \alpha} \Pi) \Pi \bigg] \\
=& (-1)^n \bigg[  \rho \Pi^{\sharp} (\alpha) \cdot (\iota_{d_A \beta} \Pi)  -  \rho \Pi^{\sharp} (\beta) \cdot (\iota_{d_A \alpha} \Pi)   \bigg] \Pi.
\end{align*}
{\bf Step 3}: It remains to prove that the bracket satisfies the Leibniz identity. Note that
\begin{align*}
& [ \alpha ,[ \beta,  \gamma ] ] - [ \beta, [ \alpha,  \gamma ]] \\
=& \mathcal{L}_{\Pi^{\sharp} \alpha} [ \beta,  \gamma ] + (-1)^n (\iota_{d_A \alpha} \Pi) [ \beta,  \gamma ]
- \mathcal{L}_{\Pi^{\sharp} \beta } [ \alpha,  \gamma ] - (-1)^n (\iota_{d_A \beta} \Pi) [ \alpha,  \gamma ] \\
=& \mathcal{L}_{\Pi^{\sharp} \alpha} \mathcal{L}_{\Pi^{\sharp} \beta} \gamma + (-1)^n \mathcal{L}_{\Pi^{\sharp} \alpha}  (\iota_{d_A \beta} \Pi) \gamma
+ (-1)^n (\iota_{d_A \alpha} \Pi) \mathcal{L}_{\Pi^{\sharp} \beta} \gamma + (\iota_{d_A \alpha} \Pi) (\iota_{d_A \beta} \Pi) \gamma \\
& - \mathcal{L}_{\Pi^{\sharp} \beta} \mathcal{L}_{\Pi^{\sharp} \alpha} \gamma - (-1)^n \mathcal{L}_{\Pi^{\sharp} \beta}  (\iota_{d_A \alpha} \Pi) \gamma
- (-1)^n (\iota_{d_A \beta} \Pi) \mathcal{L}_{\Pi^\sharp\alpha} \gamma - (\iota_{d_A \beta} \Pi) (\iota_{d_A \alpha} \Pi) \gamma \\
=& \mathcal{L}_{\Pi^{\sharp} [ \alpha,  \beta ] } \gamma + (-1)^n (\iota_{d_A \beta} \Pi) \mathcal{L}_{\Pi^{\sharp} \alpha} \gamma +
(-1)^n \rho \Pi^{\sharp} (\alpha) \cdot (\iota_{d_A \beta} \Pi) \gamma + (-1)^n (\iota_{d_A \alpha} \Pi) \mathcal{L}_{\Pi^{\sharp} \beta} \gamma \\
& - (-1)^n (\iota_{d_A \alpha} \Pi) \mathcal{L}_{\Pi^{\sharp} \beta} \gamma - (-1)^n \rho \Pi^{\sharp} (\beta) \cdot (\iota_{d_A \alpha} \Pi) \gamma - (-1)^n (\iota_{d_A \beta} \Pi) \mathcal{L}_{\Pi^{\sharp} \alpha} \gamma \\
=& \mathcal{L}_{\Pi^{\sharp} [ \alpha,  \beta ] } \gamma + (-1)^n \bigg[  \rho \Pi^{\sharp} (\alpha) \cdot (\iota_{d_A \beta} \Pi) - \rho \Pi^{\sharp} (\beta) \cdot (\iota_{d_A \alpha} \Pi)  \bigg] \gamma\\
=& \mathcal{L}_{\Pi^{\sharp} [ \alpha,  \beta ] } \gamma + (-1)^n (\iota_{d_A [ \alpha,  \beta ]} \Pi) \gamma  \hspace*{1cm} (\text{using}~~ {\bf Step ~~2}) \\
=& [ [ \alpha,  \beta ], \gamma ].
\end{align*}
\end{proof}

\begin{remark}\label{leibniz-algb-associted-base}
Let $(A, \Pi)$ be a Lie algebroid over $M$ with a Nambu structure of order $n$. Note that the Nambu-Poisson tensor  $\Lambda \in \Gamma(\wedge^nTM)$ corresponding to the bracket $\{~, \ldots, ~\}_\Pi$ as obtained in Proposition \ref{induced-NP-structure-base-prop} is given by
\begin{align*}
\Lambda (df_1, \ldots, df_{n}) & =  \{f_1, \ldots, f_n\}_{\Pi} = \Pi (d_A f_1, \ldots, d_A f_n)\\
& = \Pi (\rho^*df_1, \ldots, \rho^*df_n)= (\wedge^n\rho)(\Pi)(df_1, \ldots ,df_n).
\end{align*}
Thus,  $\Lambda = \wedge^n \rho (\Pi).$ Moreover, note that the anchor of the Leibniz algebroid associated to the Lie algebroid $TM$ with the Nambu structure $\Lambda$ (cf. Proposition \ref{leib-algbd}) is given by
$$\Lambda^\sharp = (\rho \circ \Pi^\sharp) \circ \wedge^{n-1}\rho^\ast.$$
\end{remark}

\begin{remark}\label{leibniz algebra-comparison}
\begin{enumerate}
\item If $(M, \Lambda)$ is a Nambu-Poisson manifold of order $n,$ then the Leibniz algebroid structure on $\wedge^{n-1}T^\ast M$ associated to the tangent Lie algebroid with the Nambu structure $\Lambda$ (cf. Example \ref{lie-nam-exam}(1)) as obtained above reduces to that obtained in \cite{ibanez}. 

\item Observe that having defined the notion of a Nambu structure on a Lie algebroid, the author \cite{wade} showed that there is a Leibniz algebroid structure on the vector bundle $\wedge^{n-1}A^\ast$ associated to any Nambu structure $\Pi$ of order $n$ on a Lie algebroid $A,$ where the Leibniz algebra bracket on $\Gamma(\wedge^{n-1}A^\ast)$ and the anchor are given, respectively, by 
$$[\alpha, \beta] = \mathcal L_{\Pi^\sharp\alpha}\beta - \iota _{\Pi^\sharp\beta}d_A\alpha~~~~\mbox{and}~~~~\rho\circ \Pi^\sharp,$$ for all $\alpha,~\beta \in \Gamma(\wedge^{n-1}A^\ast).$ (Note that this bracket differs from the bracket (\ref{leibniz-bracket}) only in the second term.)
However, in general, if $\Pi$ is a Nambu structure on a Lie algebroid $A$ both in the sense of the present paper and that of \cite{wade}, then the associated Leibniz algebroid structures on $\wedge^{n-1}A^\ast$ may be different. For instance, consider the following case. Let $M = \mathbb R^{n+k}, ~k\geq 1$ with the Nambu structure  $\Pi$ of order $n$ given by
$$\Pi = \frac{\partial}{\partial x_1}\wedge \cdots \wedge\frac{\partial}{\partial x_n}.$$  Let $\alpha = x_1 dx_2 \wedge \cdots \wedge dx_n$  and $\beta = dx_3\wedge \cdots \wedge dx_{n+1}.$ We show that the two brackets of $\alpha$ and $\beta$ are different. For this it is enough to show that they differ in the second term. For these $\alpha$ and $\beta,$ the second term in our case is $(-1)^n(\iota_{d\alpha}\Pi) \beta = (-1)^n dx_3\wedge \cdots \wedge dx_{n+1}\neq 0,$ whereas, the second term according to the above definition of \cite{wade} is $-\iota_{\Pi^\sharp\beta}d\alpha = 0.$
\end{enumerate}  
\end{remark}

\begin{remark}\label{loday-algebroid}
It may be remarked that in \cite{grabowski-poncin}, the authors introduced a notion of Loday algebroid which is more geometric than Leibniz algebroid and appears in many places \cite{jubin-poncin-uchino}. 
Recall that a Loday algebroid is a (left) Leibniz algebroid $(A, [~,~], \rho)$ together with a derivation
$$D : C^\infty(M) \rightarrow \mbox{Hom}_{C^\infty(M)}(\Gamma (A^{\otimes 2}), \Gamma A),$$ such that
$$[fX, Y] = f[X, Y] -(\rho(Y)f)X + D(f)(X, Y), ~~X, Y \in \Gamma A ~~\mbox{and}~~ f \in C^\infty(M).$$
Lie algebroids, Courant algebroids, Leibniz algebroids associated to  Nambu-Poisson manifolds are examples of Loday algebroids. In fact, the Leibniz algebroid $(\wedge^{n-1}A^*, [ ~,~ ], \rho \circ \Pi^{\sharp})$ of Proposition \ref{leib-algbd} associated to a Lie algebroid $A$ with a Nambu structure $\Pi$ turns out to be a Loday algebroid, where the derivation $D$ is given by
$$D(f)(\alpha, \beta) = \rho\circ\Pi^\sharp(\beta)(f)\alpha - \rho\circ\Pi^\sharp(\alpha)(f)\beta + d_Af\wedge \iota_{\Pi^\sharp(\alpha)}\beta ,$$ for all $f\in C^\infty(M)$, $\alpha, \beta \in \Gamma (\wedge^{n-1}A^*),$ $d_A$ being the Lie algebroid differential operator (see Section \ref{2}).  
\end{remark}

\begin{remark}
Let $A$ be an oriented Lie algebroid over $M$ of rank $m$ and $\Pi \in \Gamma(\wedge^mA)$ be any $m$-multisection of $A$. Then by Remark \ref{f-nambu},
$\Pi$ is a Nambu structure on $A$ of order $m$. If $\Pi \equiv 0$, then the Leibniz algebroid structure on $\wedge^{m-1}A^*$ is trivially a Lie algebroid, since the Leibniz bracket
is identically zero in this case. If $\Pi$ is non-vanishing, then the next proposition shows that the induced Leibniz algebroid is also a Lie algebroid.
\end{remark}

\begin{prop}\label{leib-lie}
 Let $A$ be an oriented Lie algebroid over $M$ of rank $m$ and $\Pi$ be a non-vanishing $m$-multisection of $A$. Then the Leibniz algebroid structure
on $\wedge^{m-1}A^*$ is a Lie algebroid.
\end{prop}

\begin{proof}
 Let $M' = \{ x \in M ~| ~~\Pi (x) \neq 0 \}.$ Then $M'$ is an open subset of $M$. Consider the restriction of the Lie algebroid $A \rightarrow M$
to the open subset $M' \subseteq M$, which we denote by $A_{M'} \rightarrow M'$. Then $\Pi$ induces a Nambu structure $\Pi_{M'}$ of order $m$ on the Lie algebroid $A_{M'}$
which is nowhere vanishing. Therefore $\Pi_{M'}$ is given by a volume element in $\Gamma(\wedge^m A^*_{M'})$, and the induced homomorphism
\begin{align*}
 \Pi^{\sharp}_{M'} : \Gamma(\wedge^{m-1}A^*_{M'}) \rightarrow \Gamma(A_{M'})
\end{align*}
is an isomorphism.

To prove that the Leibniz algebroid structure on $\wedge^{m-1}A^*$ is a Lie algebroid, let $\alpha, \beta \in \Gamma(\wedge^{m-1}A^*)$ and consider
$\sigma  = [\alpha, \beta] + [\beta, \alpha]$. Now,
\begin{align*}
 \Pi^{\sharp}_{M'} (\sigma|_{M'}) = \Pi^{\sharp}(\sigma) |_{M'} = ([\Pi^{\sharp} \alpha, \Pi^{\sharp} \beta] + [\Pi^{\sharp} \beta, \Pi^{\sharp} \alpha])|_{M'} = 0.
\end{align*}
Hence  $\sigma|_{M'} = 0$. By continuity, $\sigma$ is zero on the boundary of $M'$. It is easy to check that, $\sigma$ is zero outside the closure of $M'$. Thus $\sigma \equiv 0$ and hence the bracket is skew-symmetric.
\end{proof}

\begin{corollary}
 Let $A$ be an oriented Lie algebroid of rank $m$, $m \geq 3$, and $\mu \in \Gamma(\wedge^mA^*)$ be a nowhere vanishing section. Then the Leibniz algebroid associated to the
Nambu tensor $\Pi_\mu$ (cf. Proposition \ref{volume-nambu}) is a Lie algebroid.
\end{corollary}

\begin{prop}
 Let $(A, [~,~], \rho)$ be a Lie algebroid over $M$ and $\Pi \in \Gamma(\wedge^n A)$ be a Nambu structure of order $n$ $(n \geq 3)$. Then the map
$\wedge^{n-1} \rho^\ast : \wedge^{n-1}T^\ast M \rightarrow \wedge^{n-1} A^\ast $ induced by the anchor defines a morphism of Leibniz algebroids where $M$ is considered as a Nambu-Poisson manifold with the Nambu tensor $\Lambda = \wedge^n \rho  (\Pi)$.
\end{prop}

\begin{proof}
As $\Lambda = \wedge^n \rho (\Pi)$, we have $\Lambda^{\sharp} = \rho \circ \Pi^{\sharp} \circ \wedge^{n-1} \rho^\ast.$ Therefore the given map commutes with the anchor maps.
Let $\alpha, \beta \in \Omega^{n-1}(M).$ Then it is straightforward to check that
\begin{enumerate}
\item $ \wedge^{n-1} \rho^{*} (\mathcal{L}_{\Lambda^{\sharp} \alpha} \beta) = \mathcal{L}_{ \Pi^{\sharp}(\wedge^{n-1} \rho^{*} (\alpha))} {\wedge^{n-1} \rho^{*} (\beta)} $ and
\item $\iota_{d \alpha} \Lambda = \iota_{\wedge^n \rho^{*} (d \alpha)} \Pi = \iota_{d_A (\wedge^{n-1} \rho^{*} (\alpha))} \Pi$.
\end{enumerate}
Therefore,
\begin{align*}
 \wedge^{n-1} \rho^{*} [\alpha, \beta] =& \wedge^{n-1} \rho^{*} \big( \mathcal{L}_{\Lambda^{\sharp} \alpha} \beta + (-1)^n (\iota_{d \alpha} \Lambda) \beta  \big)\\
=&  \mathcal{L}_{ \Pi^{\sharp}(\wedge^{n-1} \rho^\ast (\alpha))} {\wedge^{n-1} \rho^\ast (\beta)} + (-1)^n (\iota_{d \alpha} \Lambda)  \wedge^{n-1} \rho^\ast (\beta) \\
=&  \mathcal{L}_{ \Pi^{\sharp}(\wedge^{n-1} \rho^\ast (\alpha))  } {\wedge^{n-1} \rho^{*} (\beta)} + (-1)^n  (\iota_{d_A (\wedge^{n-1} \rho^{*} (\alpha))} \Pi) \wedge^{n-1} \rho^{*} (\beta) \\
=&  [\wedge^{n-1} \rho^{*} (\alpha), \wedge^{n-1} \rho^\ast (\beta)].
\end{align*}
Thus, $\wedge^{n-1} \rho^{*}$ is a Leibniz algebroid morphism.
\end{proof}

\begin{remark}\label{singular-foliation-lie-algb-nambu-structure}
We have remarked before (cf. Remark \ref{singular-foliation-NP}) that there is a singular foliation associated to any Nambu-Poisson manifold. Moreover, it is known that given a Leibniz algebroid over a smooth manifold $M$, the image of the anchor defines an integrable distribution on $M$ \cite{hagiwara}.  Suppose $(A, \Pi)$ is a Lie algebroid over $M$ with a Nambu structure of order $n$ and let $\Lambda$ be the induced Nambu-Poisson structure on $M$. Then we have Leibniz algebroids $(\wedge^{n-1}A^*, [ ~,~ ], \rho \circ \Pi^{\sharp})$ and $(\wedge^{n-1}T^*M, [ ~,~ ], \Lambda^{\sharp})$ (cf. Proposition \ref{leib-algbd}) associated to the Lie algebroids with Nambu structure $(A, \Pi)$ and $(TM , \Lambda = \wedge^n \rho (\Pi)),$ respectively.  Note that the anchor $\Lambda^\sharp$ is given by (cf. Remark \ref{leibniz-algb-associted-base})
$$\Lambda^\sharp = (\rho \circ \Pi^\sharp) \circ \wedge^{n-1}\rho^\ast.$$ It follows that if $\rho$ is injective then the distribution on $M$ induced from the Leibniz algebroid $\wedge^{n-1}A^*$ coincides with the characteristic distribution on $M$ associated to the Nambu structure $\Lambda.$ In general, Image $(\Lambda^\sharp)_m ~\subseteq ~~\text{Image}~~(\rho \circ\Pi^\sharp)_m,$ for all $m \in M.$
\end{remark} 

\section{modular class}\label{4}
Recall that the modular class  of a Nambu-Poisson manifold $M$ of order $n$ with associated Nambu tensor $\Lambda$ was introduced in \cite{ibanez}. In this section, we introduce the notion of modular class of a Lie algebroid $A$ with a Nambu structure $\Pi$ of order $n>2,$ generalizing the classical case.

Let $(A, \Pi)$ be a Lie algebroid with a Nambu structure $\Pi \in \Gamma(\wedge^nA)$ of order $n$. Then by Proposition \ref{leib-algbd}, the space $\Gamma(\wedge^{n-1}A^*)$ of sections of the bundle $\wedge^{n-1}A^*$ is a Leibniz algebra with bracket $[~,~]$ is given by $[ \alpha, \beta ] = \mathcal{L}_{\Pi^{\sharp} \alpha} \beta + (-1)^n (\iota_{d_A \alpha} \Pi) \beta,$
for all $\alpha, \beta \in \Gamma(\wedge^{n-1}A^*).$ The modular class of $(A, \Pi)$ will be introduced
as an element in the first cohomology group of the Leibniz algebra cohomology of $(\Gamma(\wedge^{n-1}A^*) , [~,~])$ with coefficients in $C^\infty(M)$.
\begin{defn}\label{mod-Def}
Let $(A, \Pi)$ be an oriented Lie algebroid with a Nambu structure of order $n$. Suppose $\mu \in \Gamma(\wedge^\text{top}A^*)$ is a nowhere vanishing element representing the orientation. The
{\it modular tensor field} associated with $\mu$ is denoted by $M^\mu \in \Gamma(\wedge^{n-1}A)$ and is defined by the following relation
\begin{align}\label{equation-defn-modular-field}
 (\iota_{\alpha} M^\mu ) \mu = \mathcal{L}_{\Pi^{\sharp} \alpha} \mu + (-1)^n (\iota_ {d_A \alpha } \Pi) \mu,
\end{align}
for all $\alpha \in \Gamma(\wedge^{n-1}A^*).$
\end{defn}

\begin{remark}\label{modular-tensor}
To motivate the above defining condition of a modular tensor field, let us look at the classical case. Let $M$ be an oriented Nambu-Poisson manifold of order $n$ with associated Nambu tensor $\Lambda.$ Recall that the modular tensor field of $M$ corresponding to a given volume form $\eta$ of $M$ is defined as follows \cite{ibanez}. Consider the mapping 
$$\mathcal M^\eta : C^\infty(M) \times \cdots \times C^\infty(M) \rightarrow C^\infty(M),$$ defined by 
$$\mathcal{L}_{X_{f_1 ...f_{n-1}}} \eta = \mathcal M^\eta (f_1, \ldots, f_{n-1}) \eta, \hspace{0.05cm}f_1, \ldots, f_{n-1} \in C^\infty(M).$$ This is skew-symmetric and satisfies derivation property in each argument with respect to product of functions and hence defines an $(n-1)$-vector field on $M$ which is by definition the modular tensor field $M^\eta$ on $M$. Note that we may view $TM$ as a Lie algebroid equipped with the Nambu structure $\Lambda$ (see Example \ref{lie-nam-exam} (1)). It is then natural to expect that our definition of modular tensor field of $(TM, \Lambda)$ should be the same as the one defined in \cite{ibanez} for the Nambu-Poisson manifold $M$. It is indeed the case. 

To see this, observe from Definition \ref{mod-Def} that the modular tensor field $M^\eta \in \Gamma(\wedge^{n-1}TM)$ associated with a given volume form $\eta$ is given by
\begin{align*}
 (\iota_{\alpha} M^\eta ) \eta = \mathcal{L}_{\Lambda^{\sharp} \alpha} \eta + (-1)^n (\iota_ {d \alpha } \Lambda) \eta,
\end{align*}
for all $\alpha \in \Omega^{n-1}(M)$. Let $f_1, \ldots, f_{n-1} \in C^\infty(M)$ and take $\alpha = df_1 \wedge \cdots \wedge df_{n-1}$. Therefore, the modular tensor field
$M^\eta$ satisfies the relation  $\mathcal{L}_{\Lambda^{\sharp}(df_1 \wedge \cdots \wedge df_{n-1})} \eta = (\iota_{df_1 \wedge \cdots \wedge df_{n-1}} M^\eta) \eta.$
In other words, $\mathcal{L}_{X_{f_1 ...f_{n-1}}} \eta = M^\eta (f_1, \ldots, f_{n-1}) \eta.$
Thus for a Nambu-Poisson manifold $(M, \Lambda)$ viewed as a Lie algebroid $(TM, \Lambda)$ with Nambu structure, the modular tensor field as defined in Definition
\ref{mod-Def} coincides with that of \cite{ibanez} for $M$.

Conversely, the modular tensor field of a Nambu-Poisson manifold $(M, \Lambda)$ as introduced in \cite{ibanez} satisfies (\ref{equation-defn-modular-field}). To prove this assertion it is enough to prove it for $\alpha \in \Gamma (\wedge^{n-1}T^\ast M)$ of the form $\alpha = gdf_1\wedge \cdots \wedge df_{n-1},$ where $g, f_1, \ldots ,f_{n-1} \in C^\infty(M).$

By Remark \ref{properties-top-degree} (3), we obtain
\begin{align*}
\mathcal L_{\Lambda^\sharp(\alpha)}\eta & = \mathcal L_{g\Lambda^\sharp (df_1\wedge \cdots \wedge df_{n-1})}\eta = g\mathcal L_{\Lambda^\sharp (df_1\wedge \cdots df_{n-1})}\eta + (\Lambda^\sharp( df_1\wedge \cdots \wedge df_{n-1})(g))\eta\\
& = g(\iota_{df_1\wedge \cdots \wedge df_{n-1}}M^\eta)\eta + (\iota_{df_1\wedge \cdots \wedge df_{n-1}\wedge dg}\Lambda)\eta = (\iota_\alpha M^\eta)\eta - (-1)^n(\iota_{d\alpha}\Lambda)\eta.
\end{align*}
Thus, Equation (\ref{equation-defn-modular-field}) is satisfied.
\end{remark}

Next we study some properties of the modular tensor field $M^\mu \in \Gamma(\wedge^{n-1}A)$ associated with $\mu \in \Gamma(\wedge^\text{top}A^*)$ of an orientable Lie algebroid equipped with a Nambu structure of order $n$.

\begin{prop}\label{modular-tensor-divergence}
The modular tensor field associated with $\mu$ is the divergence of $\Pi$ with respect to $\mu$. Explicitly,
$ M^\mu = \partial_\mu (\Pi),$ where $\partial_\mu$ is the divergence operator introduced in Section \ref{2}.
\end{prop}

\begin{proof}
From Proposition \ref{formula} we have $ \iota_{\alpha} \partial_\mu (\Pi) = \text{div}_\mu (\Pi^{\sharp} \alpha) + (-1)^n \iota_{d_A \alpha} \Pi,$
for all $\alpha \in \Gamma(\wedge^{n-1}A^*)$, which implies that
$ (\iota_{\alpha} \partial_\mu (\Pi)) \mu = \mathcal{L}_{\Pi^{\sharp} \alpha} \mu + (-1)^n (\iota_{d_A \alpha} \Pi) \mu.$ Hence the result.
\end{proof}

\begin{corollary}
The modular tensor field associated with $\mu$ satisfies
$ \partial_\mu (M^\mu) = 0.$
\end{corollary}

\begin{corollary}
 The modular tensor field associated to $\mu$ satisfies
$$ \iota_{M^\mu} \mu = d_A (\iota_\Pi \mu), \hspace*{0.5cm} \iota_{M^\mu} \mu = \partial_\Pi \mu,$$
 where $\partial_\Pi = [d_A, \iota_\Pi]$.
\end{corollary}

\begin{proof}
Observe that $ M^\mu = \partial_\mu (\Pi) = {*_\mu}^{-1} \circ d_A \circ *_\mu (\Pi) = {*_\mu}^{-1} (d_A \iota_\Pi \mu).$
As a consequence, we get $\iota_{M^\mu} \mu = *_\mu (M^\mu) = d_A (\iota_\Pi \mu) = [d_A, \iota_\Pi] \mu.$
\end{proof}
Thus, the modular tensor field vanishes if and only if $\iota_\Pi \mu$ is $d_A$-closed.

\begin{prop}
 The modular tensor field $M^\mu$ associated with $\mu$ satisfies
$ \mathcal{L}_{M^\mu} \mu = 0,$
where $\mathcal{L}_P = [\iota_P, d_A]$ is the generalized Lie derivative with respect to $P \in \Gamma(\wedge^{\bullet}A).$
\end{prop}

\begin{proof}
This follows because
$$\mathcal{L}_{M^\mu} \mu = [\iota_{M^\mu}, d_A] \mu = (\iota_{M^\mu} \circ d_A - (-1)^{n-1} d_A \circ \iota_{M^\mu}) \mu = 0.$$
\end{proof}

Let $(A, \Pi)$ be an oriented Lie algebroid with a Nambu structure of order $n$. Consider the associated Leibniz algebra
$\mathcal{A} = (\Gamma(\wedge^{n-1}A^*) , [~,~])$ with trivial representation on $C^\infty (M)$ given by
$$ \Gamma(\wedge^{n-1}A^*) \times  C^\infty(M) \rightarrow C^\infty(M),~~(\alpha, f) \mapsto (\rho \Pi^{\sharp}(\alpha))f.$$

Let $\mu \in \Gamma(\wedge^\text{top}A^*)$ be a nowhere vanishing element. Then the modular tensor field $M^\mu$ associated with $\mu$ defines a map (also denoted by the same symbol)
$$ M^\mu : \Gamma(\wedge^{n-1}A^*) \rightarrow C^\infty(M),~~\alpha \mapsto \iota_{\alpha} M^\mu,$$
for all $\alpha \in \Gamma(\wedge^{n-1}A^*) $. Then we have the following result.

\begin{prop}
The map $M^\mu : \Gamma(\wedge^{n-1}A^*) \rightarrow C^\infty(M)$ defines a $1$-cocycle in the Leibniz algebra cohomology of $\mathcal{A} = (\Gamma(\wedge^{n-1}A^*) , [~,~])$ with coefficients in $C^\infty(M).$
\end{prop}

\begin{proof}
From the definition of $M^\mu$, we have
\begin{align*}
(\iota_{[ \alpha, \beta]} M^\mu ) \mu =& \mathcal{L}_{\Pi^{\sharp}  [\alpha, \beta ]} {\mu} + (-1)^n (\iota_{d_A [ \alpha, \beta ]} \Pi) \mu = \mathcal{L}_{[\Pi^{\sharp} \alpha, \Pi^{\sharp} \beta]} \mu +  (-1)^n (\iota_{d_A [ \alpha, \beta ]} \Pi) \mu \\
=& \mathcal{L}_{\Pi^{\sharp} \alpha} \mathcal{L}_{\Pi^{\sharp} \beta} \mu - \mathcal{L}_{\Pi^{\sharp} \beta} \mathcal{L}_{\Pi^{\sharp} \alpha} \mu +  (-1)^n (\iota_{d_A [ \alpha, \beta ]} \Pi) \mu \\
=& \mathcal{L}_{\Pi^{\sharp} \alpha} \big[\iota_\beta M^\mu - (-1)^n (\iota_{d_A \beta} \Pi)\big] \mu
 - \mathcal{L}_{\Pi^{\sharp} \beta} \big[\iota_\alpha M^\mu - (-1)^n (\iota_{d_A \alpha} \Pi)\big] \mu \\
&  + (-1)^n (\iota_{d_A [ \alpha, \beta ]} \Pi) \mu\\
=& \big[\iota_\beta M^\mu - (-1)^n (\iota_{d_A \beta} \Pi)\big] \mathcal{L}_{\Pi^{\sharp} \alpha} \mu + \rho \Pi^{\sharp} (\alpha) (\iota_\beta M^\mu) \mu - (-1)^n \rho \Pi^{\sharp} (\alpha) (\iota_{d_A \beta} \Pi) \mu\\
& - \big[\iota_\alpha M^\mu - (-1)^n (\iota_{d_A \alpha} \Pi)\big] \mathcal{L}_{\Pi^{\sharp} \beta} \mu - \rho \Pi^{\sharp} (\beta) (\iota_\alpha M^\mu) \mu + (-1)^n \rho \Pi^{\sharp} (\beta) (\iota_{d_A \alpha} \Pi) \mu\\
& + (-1)^n (\iota_{d_A [ \alpha, \beta ]} \Pi) \mu \\
=&  \bigg[\rho \Pi^{\sharp} (\alpha) (\iota_\beta M^\mu) - \rho \Pi^{\sharp} (\beta) (\iota_\alpha M^\mu)\bigg] \mu
\end{align*}
(in the last step we have used Equation (\ref{eqn2}) in the proof of Proposition \ref{leib-algbd}).
As $\mu$ is a nowhere vanishing section, we have $ \iota_{ [\alpha, \beta]} M^\mu  =  \rho \Pi^{\sharp} (\alpha) (\iota_\beta M^\mu) - \rho \Pi^{\sharp} (\beta) (\iota_\alpha M^\mu).$

It follows that the map $M^\mu$ as defined above is a $1$-cocycle in the Leibniz algebra cohomology of $\mathcal{A} = (\Gamma(\wedge^{n-1}A^*) , [~,~])$ with coefficients in $C^\infty(M).$
\end{proof}
\begin{prop}
The cohomology class $[M^\mu] \in \mathcal{H}^1_{Leib} (\mathcal{A})$ does not depend on the chosen non-vanishing section $\mu \in \Gamma(\wedge^\text{top}A^*)$.
\end{prop}
\begin{proof}
Let $\mu' \in \Gamma(\wedge^\text{top}A^*)$ be another non-vanishing element. Then there exists a nowhere vanishing
function $f \in C^\infty(M)$ such that $\mu' = f \mu.$ We may assume that $f > 0$ everywhere.

Observe that
\begin{align*}
 (\iota_{\alpha} M^{\mu'}) \mu' =& \mathcal{L}_{\Pi^{\sharp}\alpha} {\mu'} + (-1)^n (\iota_{d_A \alpha} \Pi) \mu' = f \mathcal{L}_{\Pi^{\sharp}\alpha} {\mu} + (\rho \Pi^{\sharp}(\alpha))(f) \mu + (-1)^n (\iota_{d_A \alpha} \Pi) f \mu \\
=& f (\iota_{\alpha} M^{\mu}) \mu + (\rho \Pi^{\sharp}(\alpha))(f) \mu. 
\end{align*}
Therefore,
$$ \iota_{\alpha} M^{\mu'} = \iota_{\alpha} M^{\mu} + \frac{1}{f} (\rho \Pi^{\sharp}(\alpha))(f) = \iota_{\alpha} M^{\mu} + \rho \Pi^{\sharp}(\alpha) (\text{log} f) = \iota_{\alpha} M^{\mu} + \iota_{\alpha} d_{\mathcal{A}}(\text{log} f),$$
where $d_{\mathcal{A}}$ is the coboundary of the cochain complex of the Leibniz algebra $\mathcal{A}$ with coefficients in $C^\infty(M)$. Hence we have $ M^{\mu'} = M^{\mu} + d_{\mathcal{A}}(\text{log} f).$
Thus, $M^{\mu}$ and $M^{\mu'}$ represent the same cohomology class in $\mathcal{H}^1_{Leib}(\mathcal{A}).$
\end{proof}

\begin{defn}\label{def-modular-class}
 Let $(A, \Pi)$ be an oriented Lie algebroid with a Nambu structure of order $n$. Then the cohomology class of the modular tensor fields lying in
$\mathcal{H}^1_{Leib} (\mathcal{A})$ is called the {\it modular class} of $(A, \Pi).$ A Lie algebroid with Nambu structure is called {\it unimodular} if its modular class vanishes.
\end{defn}

\begin{remark}
To define modular class of a Lie algebroid equipped with a nambu structure we have assumed that the Lie algebroid is oriented. However, this assumption is not essential as one can use the notion of density to define modular tensor field and modular class of a Lie algebroid with a Nambu structure which is not oriented.
\end{remark}

It is well-known \cite{wein3} that the modular class of Poisson manifold can be interpreted as an obstruction to the existence of a density invariant under the flows of all Hamiltonian vector fields. In the case of the Lie algebroid associated with a foliation with an orientable normal bundle, the modular class is the obstruction to the existence of an invariant transverse measure \cite{yks2}. The following is a result in the same spirit.

\begin{prop}\label{modular-class-obstruction}
Let $(A, \Pi)$ be a Lie algebroid with a Nambu structure $\Pi$ of order $n$. If the modular class of $(A, \Pi)$ is zero, then there exists a density of $A,$ invariant under all Hamiltonian $A$-sections (cf. Definition \ref {hamiltonian-section}). The Converse is true if $\Gamma A^\ast$ is locally generated by elements of the form $d_Af,~~f\in C^\infty(M).$ 
\end{prop}     

\begin{proof}
We prove the result assuming the Lie algebroid $A$ is orientable as a vector bundle. Let $\mu \in \Gamma(\wedge^{\text{top}}A^\ast)$ defines the orientation. Let $M^\mu$ be the modular tensor field associated to $\mu.$ If $A$ is not orientable then one may work with the density bundle $\wedge^{\text{top}}A^\ast\otimes \mathcal O$ instead, where $\mathcal O$ is the orientation bundle of $M$.  Assume that the modular class of $(A, \Pi)$ is zero. Then $M^\mu = d_{\mathcal A}(f)$ for some $f \in C^\infty(M),$ where $d_{\mathcal A}$ is the coboundary of the Leibniz algebra complex of the Leibniz algebra $\mathcal{A} = (\Gamma(\wedge^{n-1}A^*) , [~,~]).$ Set $\mu^\prime = e^{(-f)}\mu.$ Then $\mu^\prime$ is nowhere vanishing and 
$$M^{\mu^\prime} = M^\mu + d_{\mathcal A}(\text{log} (e^{(-f)})) =  M^\mu - d_{\mathcal A}(f) =0.$$
It follows from the defining condition (\ref{equation-defn-modular-field}) of the modular tensor field that
$$\mathcal{L}_{\Pi^{\sharp} \alpha} \mu^\prime + (-1)^n (\iota_ {d_A \alpha } \Pi) \mu^\prime = 0,$$ for all $\alpha \in \Gamma (\wedge^{n-1} A^\ast).$ In particular, for $\alpha = d_Af_1 \wedge \cdots \wedge d_Af_{n-1}, ~~~f_1, \ldots, f_{n-1} \in C^\infty(M),$ we get $\mathcal{L}_{\Pi^{\sharp}(d_Af_1 \wedge \cdots \wedge d_Af_{n-1})} \mu^\prime = 0.$ Thus $\mu^\prime$ is invariant under all Hamiltonian $A$-sections.

Conversely, assume that $\Gamma A^\ast$ is locally generated by elements of the form $d_Af,~~f\in C^\infty(M).$ Suppose there exists an invariant density $\mu$ of the Lie algebroid $A$.\\
Let $f_1, \ldots, f_{n-1} \in C^\infty(M).$ Then from (\ref{equation-defn-modular-field}) we see that for $\alpha = d_Af_1 \wedge \cdots \wedge d_Af_{n-1},$
$$\iota_{d_Af_1 \wedge \cdots \wedge d_Af_{n-1}}M^\mu = 0.$$ 
The assumption then implies that $\iota_{\alpha}M^\mu = 0,$ for all $\alpha \in \Gamma(\wedge^{n-1}A^\ast).$ It follows that the modular tensor field $M^\mu = 0$ and hence the modular class is zero.
\end{proof}

Next, we compute the modular tensor field and the modular class of some cases.

\begin{exam}
\begin{enumerate}
\item Let $(M, \{~,\ldots ,~\})$  be an oriented Nambu-Poisson manifold with associated Nambu tensor $\Lambda.$ Then it follows from Remark \ref{modular-tensor} that the modular class of $M$ as a Nambu-Poisson manifold (\cite{ibanez}) is the same as the modular class of $(TM, \Lambda)$ considered as a
(tangent) Lie algebroid with Nambu structure.

\item Let $(\mathfrak g, [~, ~])$ be a Lie algebra of dimension $m.$ Consider $\mathfrak g$ as a Lie algebroid over a point. Let $X_1, X_2, \ldots, X_n \in \mathfrak g$ be linearly independent and $[X_i, X_j] =0,$ for all $i, j \in \{1, \ldots ,n\}, ~n < m.$ Then we know that $\Pi = X_1 \wedge \cdots \wedge X_n$ is a Nambu structure on $\mathfrak g$ (cf. Example \ref{lie-nam-exam} (2)). Extend $\{X_1, \ldots ,X_n\}$ to a basis $\{X_1, \ldots ,X_n, X_{n+1}, \ldots X_m\}$ of $\mathfrak g.$ Let $C^k_{i, j}$ be the structure constants with respect to this basis so that $$[X_i, X_j] = \sum_{k=1}^mC^k_{i, j}X_k.$$ Clearly, $C^k_{i, j} = 0$ for $1 \leq i, ~j \leq n.$ Observe that the modular tensor field $M^\mu \in \wedge^{n-1}\mathfrak g$ associated with 
$$\mu = X^\ast_1\wedge \cdots \wedge X^\ast_n\wedge X^\ast_{n+1}\wedge \cdots \wedge X^\ast_m \in \wedge^m\mathfrak g^\ast$$ is given by
$(\iota_\alpha M^\mu)\mu = \mathcal L_{\Pi^\sharp\alpha}\mu, ~~\alpha \in \wedge^{n-1}\mathfrak g^\ast.$ This follows from the defining condition (\ref{equation-defn-modular-field}) of the modular tensor field and the fact that $\iota_{\delta\alpha}\Pi = 0,$ for all $\alpha \in \wedge^{n-1}\mathfrak g^\ast.$ 

If $\alpha = X^\ast_1\wedge \cdots \wedge \widehat{X^\ast_i} \wedge \cdots \wedge X^\ast_n,~~~~  i \in \{1, \ldots ,n\},$ then,
\begin{align*}
\mathcal L_{\Pi^\sharp\alpha}\mu & = (-1)^{n-i}\mathcal L_{X_i}\mu\\
& = (-1)^{n-i}\delta \iota_{X_i}X^\ast_1\wedge \cdots \wedge X^\ast_n\wedge X^\ast_{n+1}\wedge \cdots \wedge X^\ast_m \\
& = (-1)^{n-i}(-1)^{i-1}\delta(X^\ast_1\wedge \cdots \wedge \widehat{X^\ast_i}\wedge \cdots \wedge X^\ast_n\wedge X^\ast_{n+1}\wedge \cdots \wedge X^\ast_m)\\
& = (-1)^{n-1}\big [(-1)^{n-1}X^\ast_1\wedge \cdots \wedge \widehat{X^\ast_i}\wedge \cdots \wedge X^\ast_n\wedge \delta(X^\ast_{n+1})\wedge \cdots \wedge X^\ast_m\\
& \hspace*{1cm} + (-1)^n X^\ast_1\wedge \cdots \wedge \widehat{X^\ast_i}\wedge \cdots \wedge X^\ast_n\wedge X^\ast_{n+1}\wedge \delta(X^\ast_{n+2}) \wedge \cdots \wedge X^\ast_m \\
& \hspace*{1cm}\vdots \\
& \hspace*{1cm}+ (-1)^{m-2} X^\ast_1\wedge \cdots \wedge \widehat{X^\ast_i}\wedge \cdots \wedge X^\ast_n\wedge X^\ast_{n+1}\wedge \cdots \wedge \delta (X^\ast_m)\big ]\\
& = (-1)^{n-1}\big [ -(-1)^{n-1}(-1)^{n-i} C^{n+1}_{i, n+1} X^\ast_1\wedge \cdots \wedge X^\ast_m\\
& \hspace*{1cm}-(-1)^n(-1)^{n-i+1}C^{n+2}_{i, n+2} X^\ast_1\wedge \cdots \wedge X^\ast_m\\
&  \hspace*{1cm}\vdots \\
&  \hspace*{1cm}-(-1)^{m-2}(-1)^{n-i+m-n-1}C^m_{i, m} X^\ast_1\wedge \cdots \wedge X^\ast_m\big ]\\
& = (-1)^{n-i +1}\big[C^{n+1}_{i, n+1} + \cdots  + C^m_{i, m}\big ]\mu.    
\end{align*}
If $\alpha = X^\ast_{j_1}\wedge \cdots \wedge X^\ast_{j_{n-1}},$ for some $j_k \notin \{1, \ldots ,n\}.$ Then $\Pi^\sharp\alpha = 0$ and hence $\iota_\alpha M^\mu = 0.$ Thus the modular tensor field $M^\mu \in \wedge^{n-1}\mathfrak g,$ considered as a linear map $\wedge^{n-1}\mathfrak g^\ast \rightarrow \mathbb R,$ is given as follows:
$M^\mu(X^\ast_{j_1}\wedge \cdots \wedge X^\ast_{j_{n-1}})$ is zero if $\{j_1, \ldots, j_{n-1}\}$ is not a subset of $\{1, \ldots, n\}$ and is equal to $(-1)^{n-i +1}(C^{n+1}_{i, n+1} + \cdots  + C^m_{i, m})$ if $(j_1, \ldots, j_{n-1}) = (1, \ldots , \hat{i}, \ldots, n),~~ 1 \leq i \leq n$.

\item Let $A$ be an oriented Lie algebroid of rank $m$ and $\mu \in \Gamma(\wedge^m A^*)$ be a nowhere vanishing section representing the 
orientation of $A$. Let $\Pi_\mu \in \Gamma(\wedge^mA)$ be the Nambu structure of order $m$ associated to the volume element $\mu \in \Gamma(\wedge^m A^*)$, as described in Proposition \ref{volume-nambu}. Note that by Proposition \ref{modular-tensor-divergence},
the modular tensor field $M^{\mu}$ associated with $\mu$ is identically zero, since $ M^{\mu} = \partial_\mu (\Pi_\mu) = *_\mu^{-1} \circ d_A \circ *_\mu (\Pi_\mu) = 0,$ as $*_\mu(\Pi_\mu)$ is identically $1$ on $M.$ Therefore, $(A, \Pi_\mu)$ is unimodular.
\item Consider the Nambu structure $\Pi_f = f \Pi_\mu$ of order $m$ as described in Remark \ref{f-nambu}. Then the modular tensor field of $\Pi_f$ associated to $\mu$ is given by
$\Pi_\mu^{\sharp}(d_Af).$
\item (Modular tensor fields of subordinate Nambu structures): Let $(A, \Pi)$ be a Lie algebroid with a Nambu structure $\Pi$ of order $n$ and $\widetilde{\Pi} = \iota_{\bar{\alpha}}\Pi$ be a subordinate Nambu structure of order $n-k \geq 3$ as described in Example \ref{sub}. 
Suppose $\mu \in \Gamma(\wedge^\text{top} A^*)$ is a nowhere vanishing section representing the 
orientation of $A$. Let $M^{\mu}_{\Pi} \in \Gamma(\wedge^{n-1}A)$ and $M^{\mu}_{\widetilde{\Pi}} \in \Gamma(\wedge^{n-k-1}A)$ denote the modular tensor fields of $\Pi$ and $\widetilde{\Pi}$ respectively associated with $\mu.$ Then from the definition of $M^{\mu}_{\widetilde{\Pi}}$ it follows that for any $\beta \in \Gamma(\wedge^{n-k-1}A^*)$,
\begin{align*}
 (\iota_{\beta} M^{\mu}_{\widetilde{\Pi}}) \mu =& \mathcal{L}_{{\widetilde{\Pi}}^{\sharp} \beta} \mu + (-1)^{n-k} (\iota_{d_A \beta} \widetilde{\Pi}) \mu \\
=& \mathcal{L}_{{\Pi}^{\sharp} (\bar{\alpha} \wedge \beta)} \mu + (-1)^{n-k} (\iota_{d_A \beta} \iota_{\bar{\alpha}}{\Pi}) \mu \\
=& \mathcal{L}_{{\Pi}^{\sharp} (\bar{\alpha} \wedge \beta)} \mu + (-1)^{n-k} (\iota_{\bar{\alpha} \wedge d_A \beta} \Pi) \mu\\ 
=& \mathcal{L}_{{\Pi}^{\sharp} (\bar{\alpha} \wedge \beta)} \mu + (-1)^n (\iota_{d_A(\bar{\alpha} \wedge \beta)} \Pi) \mu \hspace*{1cm} \text{(since $d_A \bar{\alpha} = 0$)}\\
=& (\iota_{\bar{\alpha} \wedge \beta} M^{\mu}_{\Pi}) \mu = (\iota_\beta \iota_{\bar{\alpha}} M^{\mu}_{\Pi}) \mu.
\end{align*}
Thus, the modular tensor fields are related by $M^{\mu}_{\widetilde{\Pi}} = \iota_{\bar{\alpha}} M^{\mu}_{\Pi}$.
\end{enumerate}
\end{exam}

\section{maximal nambu structure and modular class}\label{5}
In \cite{elw}, the authors introduced a notion of characteristic class of a Lie algebroid $A$ with a representation on a line bundle $L$ and used it to define the notion of modular class of a Lie algebroid. The aim of this final section is to show that for a large class of Lie algebroids with Nambu structures, more specifically, for parallelizable manifolds with maximal Numbu structures, the notion of modular class is closely related to the notion of modular class introduced in \cite{elw}. First, we briefly recall the definition of the modular class of a Lie algebroid. 
 
Let $\nabla : \Gamma A \times \Gamma L \longrightarrow \Gamma L$ be a representation of $A$ on a line bundle $L$. If $L$ is trivial and $s \in \Gamma{L}$ is a nowhere vanishing section of $L$ then define a section $ \theta_s \in \Gamma{A^*}$  by $\nabla_a s = \langle \theta_s, a \rangle s,$
for any $a \in \Gamma{A}$. It turns out that $ \theta_s$ is a $1$-cocycle in the cochain complex defining the Lie algebroid cohomology of $A$ with trivial representation. Moreover, the cohomology class represented by $ \theta_s$ is independent of the choice of $s$. This cohomology class is the characteristic class of $A$ with representation on $L$. If $L$ is non-trivial, then define the characteristic class as one half of that
of the square $L^2 = L\otimes L$ of this bundle with associated representation. For any Lie algebroid $A$, there is an intrinsic representation of $A$
on the line bundle $Q_A = \wedge^\text{top} A \otimes \wedge^\text{top}T^*M$  given by
\begin{align*}
 \nabla_a (X \otimes \nu) = [a, X] \otimes \nu + X \otimes \mathcal{L}_{\rho(a)} \nu,
\end{align*}
where $X \in \Gamma(\wedge^\text{top} A)$, $\nu \in \Gamma(\wedge^\text{top} T^*M)$, $[~,~]$ is the Gerstenhaber bracket on the multisections of $A$
and $\rho$ denotes the anchor map of the Lie algebroid $A$. The characteristic class of the Lie algebroid $A$ with the above representation on the bundle $Q_A$ is defined as the modular class of $A$. Moreover, the authors \cite{elw}  proved that the modular class of the cotangent Lie algebroid $T^\ast P$ of a Poisson manifold $P$ is twice the modular class of $P$ as a Poisson manifold defined in \cite{wein3}.

Let $A$ be an oriented Lie algebroid of rank $m$ and $\Pi$ be a non-vanishing $m$-multisection of $A$. As we have noticed before, $\Pi$ is a maximal Nambu structure on $A$ and the associated Leibniz algebroid is a Lie algebroid. In the following proposition we show that the modular class of $(A, \Pi)$ can be seen as the characteristic class of the Lie algebroid
$\mathcal{A} = \wedge^{m-1}A^*$ with a representation on the line bundle $\wedge^mA^*$, in the sense of \cite{elw}.

\begin{prop}
Let $A$ be an oriented Lie algebroid of rank $m$ and $\Pi$ be a non-vanishing $m$-multisection of $A$. Then
$ \nabla: \Gamma{\mathcal{A}} \times \Gamma(\wedge^{m}A^*) \rightarrow \Gamma(\wedge^{m}A^*) $
given by
$$ \nabla_{\alpha} \lambda = \mathcal{L}_{\Pi^{\sharp} \alpha} \lambda + (-1)^m \langle \Pi, d_A \alpha \rangle \lambda $$
for $\alpha \in \Gamma{\mathcal{A}}$, $\lambda \in \Gamma(\wedge^{m}A^*)$, defines a representation of the Lie algebroid $\mathcal{A} = \wedge^{m-1}A^*$ on the line bundle $\wedge^{m}A^*$. Moreover, the characteristic class
of the Lie algebroid $\mathcal{A} = \wedge^{m-1}A^*$ with respect to this representation is the modular class of $(A, \Pi)$.
\end{prop}

\begin{proof}
It is straightforward to verify that $\nabla$ is a representation of the Lie algebroid $\mathcal{A}$ on the line bundle $\wedge^mA^*$. The last part follows from the definition of modular class of $(A, \Pi)$.
\end{proof}

Next, assume that $M$ is an oriented manifold of dimension $m$ and $\Pi$ is a non-vanishing $m$-multivector field on $M$. Then by Remark \ref{f-nambu}, $\Pi$
is a maximal Nambu structure on the tangent Lie algebroid $TM$ and from Proposition \ref{leib-lie}, the associated Leibniz algebroid $\wedge^{m-1}T^*M$ is a Lie algebroid.
Moreover, the modular class of $(TM, \Pi),$ which is the same as the modular class of the Nambu-Poisson manifold $M,$ lies in the first Leibniz algebra cohomology of $\Gamma (\wedge^{m-1}T^*M)$. However, it is important to note that the map $ \alpha \mapsto \iota_{\alpha} M^\mu, $ defined by the modular tensor field representing this class, is actually $C^\infty(M)$-linear and hence, represents a class in the first Lie algebroid cohomology of the Lie algebroid $\mathcal{A}$ with trivial coefficients. On the other hand, the modular class of the Lie algebroid $\mathcal{A}$ defined in \cite{elw} is also a first Lie algebroid cohomology class of $\mathcal{A}.$  Thus, it is natural to compare these two notions of modular class
as elements of the first Lie algebroid cohomology of $\wedge^{m-1}T^*M$.

To compare these classes, it is enough to compare the $1$-cocycles representing them. 

Let $M$ be a parallelizable manifold of dimension $m$. We choose globally defined
$1$-forms $\eta_1, \ldots, \eta_m$ of $M$ which are linearly independent. Then $\eta = \eta_1 \wedge \cdots \wedge \eta_m$
is a volume form on $M$ and moreover,
$\bar{\eta}_i := \eta_1 \wedge \cdots \wedge \widehat{\eta}_i \wedge \cdots \wedge \eta_m,~~~~ 1\leq i \leq m$
are basis of sections of the bundle $\wedge^{m-1}T^*M.$ 

\begin{lemma}\label{lemma-final}
 For $f_1, \ldots, f_{m-1} \in C^\infty(M)$, suppose
\begin{align*}
 \mathcal{L}_{\Pi^{\sharp}(df_1 \wedge \cdots \wedge df_{m-1})} \eta_k = c_{k1}\eta_1 + \cdots + c_{km} \eta_m ,  1 \leq k \leq m
\end{align*}
where each coefficient $c_{kj}$ is a smooth function on $M$. Then
\begin{enumerate}
\item the extended Gerstenhaber bracket of the Lie algebroid $\wedge^{m-1}T^*M$ satisfies
$$[df_1 \wedge \cdots \wedge df_{m-1}, \bar{\eta}_1 \wedge \cdots \wedge \bar{\eta}_m ] = (m-1) (c_{11} + \cdots + c_{mm}) \bar{\eta}_1 \wedge \cdots \wedge \bar{\eta}_m;$$

\item the Lie derivative of the volume form satisfies
$$ \mathcal{L}_{\Pi^{\sharp}(df_1 \wedge \cdots \wedge df_{m-1})} \eta = (c_{11} + \cdots + c_{mm}) \eta.$$
\end{enumerate}
\end{lemma}

\begin{proof}
 (i) Observe that,
\begin{align*}
 &[df_1 \wedge \cdots \wedge df_{m-1}, \bar{\eta}_1 \wedge \cdots \wedge \bar{\eta}_m ] \\
=& \sum_{i=1}^m \bar{\eta}_1 \wedge \cdots \wedge [df_1 \wedge \cdots \wedge df_{m-1}, \bar{\eta}_i] \wedge \cdots \wedge \bar{\eta}_m \\
=& \sum_{i=1}^m \bar{\eta}_1 \wedge \cdots \wedge \big( \mathcal{L}_{\Pi^{\sharp}(df_1 \wedge \cdots \wedge df_{m-1})}  \bar{\eta}_i \big) \wedge \cdots \wedge \bar{\eta}_m \hspace{0.06cm}\text{(by the definition of Leibniz bracket)}\\
=& \sum_{i=1}^m \bar{\eta}_1 \wedge \cdots \wedge \big( \sum_{k=1, k\neq i}^{m}   \eta_1 \wedge \cdots \wedge  (\mathcal{L}_{\Pi^{\sharp}(df_1 \wedge \cdots \wedge df_{m-1})} \eta_k) \wedge \cdots \wedge \eta_m  \big) \wedge \cdots \wedge \bar{\eta}_m \\
=& \sum_{i=1}^m \bar{\eta}_1 \wedge \cdots \wedge\big( (\sum_{k=1, k\neq i}^{m}  c_{kk}) \bar{\eta}_i  \big) \wedge \ldots \bar{\eta}_m = \sum_{i=1}^m (\sum_{k=1, k\neq i}^{m}  c_{kk}) \bar{\eta}_1 \wedge \cdots \wedge \bar{\eta}_m \\
=& (m-1) (c_{11} + \cdots + c_{mm}) \bar{\eta}_1 \wedge \cdots \wedge \bar{\eta}_m.
\end{align*}

(ii) Next, note that
\begin{align*}
\mathcal{L}_{\Pi^{\sharp}(df_1 \wedge \cdots \wedge df_{m-1})} \eta =& \mathcal{L}_{\Pi^{\sharp}(df_1 \wedge \cdots \wedge df_{m-1})} (\eta_1 \wedge \cdots \wedge \eta_m) \\
=& \sum_{i=1}^m \eta_1 \wedge \cdots \wedge \big( \mathcal{L}_{\Pi^{\sharp}(df_1 \wedge \cdots \wedge df_{m-1})}  \eta_i \big) \wedge \cdots \wedge \eta_m \\
=& (c_{11} + \cdots + c_{mm}) \eta_1 \wedge \cdots \wedge \eta_m = (c_{11} + \cdots + c_{mm}) \eta.
\end{align*}
\end{proof}

Recall from \cite{elw} that the representation $\nabla$ of the Lie algebroid $\mathcal{A} = \wedge^{m-1}T^*M$ on the line bundle 
$$Q_{\mathcal{A}} = \wedge^\text{top} \mathcal{A} \otimes \wedge^\text{top}T^*M,$$ is given by 
\begin{align*}
 &\nabla_{df_1 \wedge \cdots \wedge df_{m-1}} ((\bar{\eta}_1 \wedge \cdots \wedge \bar{\eta}_m) \otimes \eta)\\
=& [df_1 \wedge \cdots \wedge df_{m-1}, \bar{\eta}_1 \wedge \cdots \wedge \bar{\eta}_m ] \otimes \eta + (\bar{\eta}_1 \wedge \cdots \wedge \bar{\eta}_m) \otimes \mathcal{L}_{\Pi^{\sharp}(df_1 \wedge \cdots \wedge df_{m-1})} \eta \\
=& (m-1) (c_{11} + \cdots + c_{mm}) (\bar{\eta}_1 \wedge \cdots \wedge \bar{\eta}_m) \otimes \eta + (\bar{\eta}_1 \wedge \cdots \wedge \bar{\eta}_m) \otimes \mathcal{L}_{\Pi^{\sharp}(df_1 \wedge \cdots \wedge df_{m-1})} \eta \\
& \hspace*{8cm}\mbox{(by Lemma \ref{lemma-final}(1))}\\
=& (m-1)(\bar{\eta}_1 \wedge \cdots \wedge \bar{\eta}_m) \otimes (c_{11} + \cdots + c_{mm})\eta + (\bar{\eta}_1 \wedge \cdots \wedge \bar{\eta}_m) \otimes \mathcal{L}_{\Pi^{\sharp}(df_1 \wedge \cdots \wedge df_{m-1})} \eta \\
=& m  (\bar{\eta}_1 \wedge \cdots \wedge \bar{\eta}_m) \otimes \mathcal{L}_{\Pi^{\sharp}(df_1 \wedge \cdots \wedge df_{m-1})} \eta \hspace*{2cm}\mbox{(by Lemma \ref{lemma-final} (2))}\\ 
=& m  (\bar{\eta}_1 \wedge \cdots \wedge \bar{\eta}_m) \otimes M^{\eta} (df_1 \wedge \cdots \wedge df_{m-1}) \eta\\
=& m M^{\eta} (df_1 \wedge \cdots \wedge df_{m-1}) (\bar{\eta}_1 \wedge \cdots \wedge \bar{\eta}_m) \otimes \eta.
\end{align*}

Thus, from the definition of modular class of a Lie algebroid, we have
\begin{align*}
 \theta_{(\bar{\eta}_1 \wedge \cdots \wedge \bar{\eta}_m) \otimes \eta} (df_1 \wedge \cdots \wedge df_{m-1}) = m M^{\eta} (df_1 \wedge \cdots \wedge df_{m-1}).
\end{align*}
Since both sides are $C^\infty(M)$-linear, the $(m-1)$-vector fields are related by
\begin{align*}
 \theta_{(\bar{\eta}_1 \wedge \cdots \wedge \bar{\eta}_m) \otimes \eta} = m M^{\eta}.
\end{align*}

Summarizing the above discussions we obtain the following result.
\newpage
\begin{thm}\label{modular-compare}
Let $M$ be a parallelizable manifold of dimension $m$. Let $\Pi$ be a non-vanishing $m$-vector field on $M$. Then the modular class of the 
Lie algebroid $\wedge^{m-1}T^*M$ is $m$-times the modular class of the Nambu-Poisson manifold $(M, \Pi),$ both being considered as elements of the first Lie algebroid cohomology of $\wedge^{m-1}T^*M$.
\end{thm}

\begin{corollary}
Under the hypothesis of Theorem \ref{modular-compare}, if $\Pi$ is the $m$-vector field associated to a volume form, then the modular class of the Lie algebroid
$\wedge^{m-1}T^*M$ is zero.
\end{corollary}

\begin{corollary}
Let $M$ be any closed oriented $3$-manifold and $\Pi$ be any $3$-vector field on $M$. Then the modular class of the Nambu-Poisson manifold $(M, \Pi)$ is 
$$\frac{1}{3} \times ~~\mbox{modular class of the Lie algebroid}~~\wedge^2T^*M.$$
\end{corollary}

We conclude with the following remark.

\begin{remark}\label{question}
In \cite{grabowski-poncin}, Grabowski et al. introduced the notion of Loday algebroids. We have noticed before (cf. Remark \ref{loday-algebroid}) that the Leibniz algebroid $(\wedge^{n-1}A^*, [ ~,~ ], \rho \circ \Pi^{\sharp})$ of Proposition \ref{leib-algbd} associated to a Lie algebroid $A$ with a Nambu structure $\Pi$ is a Loday algebroid. It would be interesting to formulate a notion of the modular class of Loday algebroids (or more generally, of  Leibniz algebroid) and  explore the relationship between the modular class of $(A, \Pi),$ as introduce in this paper, and the modular class of the Loday (or Leibniz) algebroid $(\wedge^{n-1}A^*, [ ~,~ ], \rho \circ \Pi^{\sharp}).$ This would be a higher order generalization of the result of Evens-Lu-Weinstein \cite{elw} for Poisson manifolds.   
\end{remark}
\providecommand{\bysame}{\leavevmode\hbox to3em{\hrulefill}\thinspace}
\providecommand{\MR}{\relax\ifhmode\unskip\space\fi MR }
\providecommand{\MRhref}[2]{%
  \href{http://www.ams.org/mathscinet-getitem?mr=#1}{#2}
}
\providecommand{\href}[2]{#2}

\end{document}